\documentclass[11pt]{amsart} 

\usepackage[utf8]{inputenc} 

\usepackage{geometry} 
\usepackage{graphicx}
\usepackage{xcolor}

\usepackage{booktabs} 
\usepackage{array} 
\usepackage{paralist} 
\usepackage{verbatim} 
\usepackage{subfig}

\usepackage{amsmath}
\usepackage{amsthm}
\usepackage{thm-restate}

\usepackage{hyperref}

\hypersetup{
    colorlinks=true,
    linkcolor=cyan,
    filecolor=blue,
    urlcolor=red,
    citecolor=orange,
}
\usepackage{cleveref}
\theoremstyle{plain}
\newtheorem{theorem}{Theorem}[section]
\newtheorem{proposition}{Proposition}[section]
\newtheorem{lemma}{Lemma}[section]

\newtheorem{definition}{Definition}[section]

\newtheorem{assumption}{Assumption}[section]

\theoremstyle{remark}

\numberwithin{equation}{section}
\numberwithin{table}{section}
\numberwithin{figure}{section}

\usepackage{amsfonts}

\begin{document}
\title{Analysis of a wavelet frame based two-scale model for enhanced edges}
\author{Bin Dong}
\address{Beijing International Center for Mathematical Research \&
Center for Machine Learning Research,
Peking University, Beijing 100871, P. R. China.}
\email{dongbin@math.pku.edu.cn}

\author{Ting Lin}
\address{School of Mathematical Sciences, Peking University, Beijing 100871, P. R. China.}
\email{lintingsms@pku.edu.cn}
\author{Zuowei Shen}
\address{Department of Mathematics, National University of Singapore, 117543, Singapore.}
\email{matzuows@nus.edu.sg}

\author{Peichu Xie}
\address{Beijing International Center for Mathematical Research,
Peking University, Beijing 100871, P. R. China.} 
\email{xie@u.nus.edu.sg}

\maketitle
\begin{abstract}

Image restoration is a class of important tasks that emerges from a wide range of scientific disciplines. It has been noticed that most practical images can be modeled as a composition from a sparse singularity set (edges) where the image contents or their gradients change drastically, and cartoon chunks in which a high degree of regularity is dominant. Enhancing edges while promoting regularity elsewhere has been an important criterion for successful restoration in many image classes. In this article, we present a wavelet frame based image restoration model that captures potential edges and facilitates the restoration procedure by a dedicated treatment both of singularity and of cartoon. Moreover, its geometric robustness is enhanced by exploiting subtle inter-scale information available in the coarse image. To substantiate our intuition, we prove that this model converges to one variant of the celebrated Mumford-Shah model when adequate asymptotic specifications are given. 

\end{abstract}
\tableofcontents
\section{Introduction}

Image restoration is a class of important tasks that emerges from a wide variety of scientific disciplines, where many restoration problems can be modeled as solving the inverse problem
\begin{equation}\label{inverse_problem}
\mathbf f=\mathbf A\mathbf u+\varepsilon,
\end{equation}
where $\mathbf f$ denotes the observed data, $\mathbf u$ denotes the unknown clean image that is to be restored, which is usually defined on a discrete pixel lattice $\mathbb O\subset\mathbb R^d$ (where we assume $d=2$ throughout this article), and $\mathbf A$ denotes an operator that characterizes the image acquisition procedure itself and/or contaminations that happen with it. For instance, $\mathbf A$ represents pixel restriction for inpainting, in particular, identity for denoising, (generalized) convolution for deblurring and line integration transform for CT imaging task. Finally, $\varepsilon$ denotes a (Gaussian) noise. The operator $\mathbf A$ often happens to be singular, therefore, in order to secure a meaningful solution, one needs to impose prior conditions on the class where clean images reside. 

Practically, one resorts to solving a regularized optimization problem of form 
\begin{equation}\label{reg_plus_fid}
\inf_{\mathbf u}\{Reg(\mathbf u)+Fid_{\mathbf f}(\mathbf u)\},
\end{equation}
where $Reg(\mathbf u)$ denotes the \emph{regularity term} that substantiates this prior on $\mathbf u$, and $Fid_{\mathbf f}(\mathbf u)=\frac{1}{2}\|\mathbf u-\mathbf f\|_2^2$ is the \emph{fidelity term} that measures the distance between $\mathbf u$ and $\mathbf f$ in the context of operator $\mathbf A$.

Historically, models based on \emph{wavelet frames} have served this purpose with remarkable success. By construction, wavelet frames are spatially localized objects that satisfy Parseval reconstruction identity but are not completely orthogonal to each other. Thanks to such flexibility as well as their intrinsic \emph{multi-resolution analytic (MRA)} structure, they have become a powerful instrument for achieving a balance between large-scale patterns and delicate local features (including sharp edges) in images. One prominent prototype is the \emph{analysis model} 
$$\inf_{\mathbf u}\left\{\lambda\|\mathbf W\mathbf u\|_{\ell^p}^p+\frac{1}{2}\|\mathbf A\mathbf u-\mathbf f\|_{\ell^2}^2\right\},$$ 
where $\mathbf W\mathbf u$ are the \emph{coefficients} obtained from the \emph{wavelet transform} $\mathbf W$, and the index $p\in(0,2)$. 

We note that in \cite{l0} the authors also addressed to the rather unusual setting with $p=0$, which was previously considered hard to solve. In particular, the norm-like procedure $\|\cdot\|_{\ell^0}$ counts the non-zero entries presenting in the respective set of coefficients, but entirely ignores their values. In this aspect, their model works more thoroughly than the traditional analysis model regarding edge restoration.

As a complementary methodology to discrete approach (on which we place the focus of this paper), variational image processing has been successful either, where one considers a function defined on a continuum domain in \eqref{reg_plus_fid}, and characterizes its regularity by variation or Sobolev norms. In \cite{tv-wf} and later in \cite{wf-pde}, the intrinsic connections between these two distinct approaches has been largely clarified. One of the key observations there is that, at finite resolution, wavelet frames (filters) supply as a versatile instrument for performing local and global time-frequency analysis, while at continuum (infinite) resolution, these are asymptotically correspondent to differential operators if certain mild conditions are imposed. Motivated by these, a number of investigations are further conducted in order to help model designing as well as to facilitate theoretical analysis.

In \cite{wf-g}, the authors formulated the following model, $$\inf_{\mathbf u,\mathbf v}\left\{\|\lambda_1\cdot(\mathbf W'\mathbf u-\mathbf v)\|_{\ell^p}^p+\|\lambda_2\cdot\mathbf W''\mathbf v\|_{\ell^q}^q+\frac{1}{2}\|\mathbf A\mathbf u-\mathbf f\|_{\ell^2}^2\right\},$$ where $\mathbf W'$ denotes the regular wavelet transform while $\mathbf W''$ denotes the \emph{wavelet packet transform} with respect to a chosen \emph{wavelet packet system}. This amounts to applying a finer time-frequency analysis to $\mathbf u$. In particular, the auxiliary variable $\mathbf v$ provides chances for the component (\emph{i.e.} partial coefficients) of $\mathbf u$ that is potentially dense, having a large $\ell^p$-sum, but can otherwise be reduced to a succinct composition of wavelet packets. This often happens in areas filled with regular cartoons. On the other hand, at points near singularity, little quantity of $\mathbf v$ can be separated out for further $\mathbf W''$ transform as it usually results in explosive coefficients. By asymptotic analysis, the authors found their discrete model energy converges to the \emph{TGV (total generalized variation) model}, a prominent variant of the classic TV-based one, when a suitable context is given.

In \cite{wf-pw}, in particular, a natural image is \emph{a priori} modeled as a {\emph{piecewise smooth function}}, and heuristically, a discrete wavelet frame based method is developed to facilitate restoration at any finite resolution. In particular, the model takes the form $$\inf_{\mathbf u}\left\{\|\lambda\cdot\mathbf W_1\mathbf u\|_{\ell^p(\Gamma^c)}+\|\mu\cdot\mathbf W_2\mathbf u\|_{\ell^1(\Gamma)}+\frac{1}{2}\|\mathbf {Au}-\mathbf f\|_{\ell^2}^2\right\},$$ where $\mathbf W_1$, $\mathbf W_2$ denote two flexibly specified wavelet transforms, $\Gamma$ denotes a pre-estimated singularity set on which an $\ell^1$-summation is done, and on the non-singular grid $\Gamma^c$ a power index $p>1$ is applied. By asymptotic analysis, on the singularity (called also a \emph{jump} or an \emph{edge}, depending on the context), the $L^1$-integrated magnitude of discontinuity is identified as a part of the limit energy, while on the region away from it, a Sobolev semi-norm is identified instead.

This model has been proven successful in the simultaneous restoration of cartoons and edges, however, some crucial issues have been left to address by that time. First of all, sharp edges, which are important features of images, cannot get sufficiently enhanced since the $\ell^1$-summed (respectively, $L^1$-integrated) cost by nature disfavors their occurrences. Secondly, the singularity location must be provided separately in advance, which leads to additional calculation burden and potential algorithmic instability.

Having these observations and practical motivations in mind, in this paper, we consider a minimization problem over the following energy functional, 
\begin{equation}
\mathbf E(\mathbf u)=\sum_{k}(\hat{\mathbf W}\mathbf u)[k]+\sum_{k}|(\mathbf A\mathbf u)[k]-\mathbf f[k]|^2,
\end{equation}
where the \emph{truncated wavelet frame operator} $\hat{\mathbf W}$ is realized as an ordinary wavelet frame transform $\mathbf W$ followed by two subsequent operations: (i) a \emph{truncated $\ell^r-$summation} over a carefully chosen (coarse-scale) neighbourhood of the lattice point $k$ in question, properly weighted in each channel; (ii) an $\ell^q-$summation over the channels. The details of its definition and relevant technicalities are presented in Subsection \ref{model_discrete}, where we focus on the special case $r=\infty$ and $q=1$. 

This model inherits the advantages usually associated with wavelet frames in aspects of time-frequency analysis, in particular, it captures cartoons and milder variations of the image nicely. Moreover, thanks to the coefficient truncation design, it tolerates sharp edges present in the image much better. Indeed, when jump size becomes significant, the energy cost (density) does not depend on it any longer. This is in a sense akin to the $\ell^0$-regularization method \cite{l0}. 
In addition, our model manages to regain isotropy systematically near singularity due to its two-scale treatment, which integrates local geometric information necessary to compensate for the innate anisotropy of any (discrete) wavelet frame system adopted. 

Further leveraging the asymptotic analysis method made available by previous studies, we exhibit that with certain particular assumptions on the choice of wavelet frames and their weights accordingly, our discrete model converges to a limit one with variational energy 

\begin{equation}
E(u)=\int_{\Omega\setminus\Gamma^*_u}|\nabla u|+\alpha\cdot\mathcal H^1(\Gamma^*_u)+\int_\Omega|Au-f|^2,
\end{equation}
where $\Gamma^*_u$ denotes the non-trivial jumps of $u$ and $\mathcal H^1(\Gamma^*_u)$ its one-dimensional Hausdorff measure. This model allows certain jumps of $u$ to be evaluated only by the length of support, not the step magnitude. Thus it particularly helps to keep (enhance) more significant jumps, and besides, treats smooth cartoon and minor jumps in an equal way, corroborating our discrete intuition. Technically, the scalar parameter $\alpha$ controls the transition between these two characteristic regimes.

{In general, efforts calculating or analyzing such models with flexible singularity often suffer some degree of technical pathology, being it discrete or continuum based. Except for various grid artifacts that are largely prevailing around curves, the unwieldy flexibility of an $SBV$ function's singularity also complicates the situation. }{In what follows, we address this problem by means of wavelet frames. We believe that our theoretical analysis has its merits in the context of singularity modeling in general, and provides insights for further implementations as well.}

It is worth note that comparing with our method, the classical Mumford--Shah model, for instance, treats minor jumps in a more rigid way, conceding to their smooth (regularized) counterpart. In image classes which do naturally contain minor jumps (\emph{e.g.} as a texture), this results in loss of information. Notable discretizations include \cite{chambolle1999finite}, in which the authors implemented a global scheme, as well as \cite{gobbino1998finite}, where an integration of finite difference scheme is successfully applied, though its context is restricted to $\Gamma$-convergence on $L^2$. Generally speaking, a completely satisfying treatment has not been achieved yet.

The rest of the paper is organized as follows. In Section \ref{sec_model_def}, we precisely specify our discrete two-scale model and in turn propose its asymptotic limit as the resolution tends to infinity. In Section \ref{proof_liminf}, we explain their working principles quantitatively. First of all, we show that in context pertaining to these particular models, any generic specimen of $SBV$ singularity can be adequately approximated by one supported on certain \emph{simple union} of curve segments (cf. Lemma \ref{curve_truncation}). Then, by rigorous theoretical analysis, we illuminate how our model's two-scale design can advance in capturing the latter by proper choice of wavelets and coefficient cut-off. In particular, we provide energy estimates near and off singularity respectively, and clarify how they exhibit distinct characteristic behaviours when resolution $n\to\infty$ and in what way they together contribute to our desired asymptotics. Finally, we construct a sequence of functions that secures the tightness of the asymptotic energy bound proposed, completing our analysis with a standard $\Gamma$-convergence characterization.

\section{{Model constructions}}\label{sec_model_def}

Here we introduce the discrete and continuum models from the two aforementioned methodologies, having in mind their respective advantages. In Subsection \ref{model_discrete} we provide a concise introduction to wavelet frames, and define our discrete model accordingly. Later, in Subsection \ref{continuum_model} we recall the definition and fundamental properties of special functions of bounded variation (SBV functions), and introduce a variational model based on SBV functions. The main contribution of our paper is the establishment of the new wavelet frame based model, and clarification of its fundamental properties by strict asymptotic analysis. In particular, we characterize the linkage between our discrete model and its proposed variational limit in terminology of $\Gamma-$convergence, and discuss some interesting special cases toward the end of this section.

\subsection{A wavelet frame based two-scale model (discrete)}\label{model_discrete}

Frames are flexible objects in a Hilbert space $\mathcal H$ that generalizes the conception of its orthonormal bases \cite{tv-wf,notes}. In particular, a finite or countable set $\{\psi_i\}_{i\in I}\subset\mathcal H$ is said to form a \emph{tight frame} of $\mathcal H$ if the identity
\begin{equation}
f=\sum_{i\in I}\left<f,\psi_i\right>\psi_i,
\end{equation}
is satisfied for any element $f\in\mathcal H$. In this article we shall focus on the special case $\mathcal H=L^2(\mathbb R^d)$ (with $d=2$ particularly).

Generally, there are many ways of constructing frames in a given Hilbert space, and as a prominent example, \emph{wavelet frames} are based on translations and dilations of certain specific (finite) collections of functions. More precisely, denote this collection as $\Psi=\{\psi_1,\cdots,\psi_J\}$, where each element $\psi_j\in L^2(\mathbb R^d)$, the \emph{quasi-affine system} that they form is defined as
\begin{equation}
X^q(\Psi)=\{\psi_{j,n,k}~|~j=1,\cdots,J;\ n\in\mathbb Z;\ k\in\mathbb Z^d\},
\end{equation}
where
\begin{equation}\label{def_trans_dil}
\psi_{j,n,k}(x)=\left\{
\begin{matrix}
2^\frac{nd}{2}\psi_j(2^nx-k), & \text{ if }n\geq0,\\
2^{nd}\psi_j(2^nx-2^nk), & \text{ if }n<0.
\end{matrix}
\right.
\end{equation}
When $X^q(\Psi)$ does form a (tight) frame in $L^2(\mathbb R^d)$, it is called a wavelet (tight) frame, which has very desirable time-frequency analytical properties for various applications. Therefore, it becomes important to find criteria so that this is theoretically guaranteed. In history, a generic family of them is formulated based on the idea of \emph{multi-resolution analysis (MRA)}, where an additional object $\psi_0:=\phi\in L^2(\mathbb R^d)$ (called a \emph{refinable function}, and $\phi_{n,k}:=\phi_{0,n,k}$ is defined by \eqref{def_trans_dil} where we allow $j=0$) is introduced together with some \emph{filters} $h_j\in\ell^2(\mathbb Z^d)$ ($j=0,1,\cdots,J$), specified so that the identity,
\begin{equation}
\psi_j(x)=2^{d}\sum_{k\in\mathbb Z^d}h_j[k]\phi(2x-k)
\end{equation}
can be verified for each $j$. Practically, there are a number of \emph{extension principles} developed in the literature, which help us find out feasible combinations of $\Psi$, $\phi$ and $h_j$, such as the \emph{unitary extension principle (UEP)}, which states that if those aforementioned functions, in particular, the filters $\{h_j\}_{j=1}^J$, are chosen such that
\begin{equation}
\sum_{j=0}^J\left|\widehat{h}_j(\xi)\right|^2=1\ \ \text{and}\ \ \sum_{j=0}^J\widehat{h}_j(\xi)\overline{\widehat{h}_j(\xi+\nu)}=0,
\end{equation}
where $\nu\in\{0,\pi\}^d\setminus(0,\cdots,0)$, then the quasi-affine system $X^q(\Psi)$ (as well as its affine counterpart) forms a tight frame. Indeed, if we know of a filter collection $\{h_j\}_{j=1}^J$ that satisfies these equations, we can reconstruct the refinable function $\phi$ as well as the wavelets $\psi_j$ with the so-called cascade algorithm, but we refrain from going into the details here.

The so-called $B$-splines are among the excellent examples of wavelet frames in $L^2(\mathbb R)$ which are obtainable from the UEP. In particular, the \emph{Haar wavelet frame} with 
\begin{equation}
\phi(x)=\left\{
\begin{matrix}
1,\ if\ x\in[0,1],\\
0,\ otherwise.
\end{matrix}
\right.
\end{equation}
\begin{equation}
h_0=\frac{1}{2}\cdot[1,1]\ \ and\ \ h_1=\frac{1}{2}\cdot[1,-1],
\end{equation}
and the \emph{piecewise linear wavelet frame} with 
\begin{equation}
\phi(x)=\left\{
\begin{matrix}
1-|x|,\ if\ x\in[-1,1],\\
0,\ otherwise.
\end{matrix}
\right.
\end{equation}
\begin{equation}
h_0=\left[\frac{1}{4},\frac{1}{2},\frac{1}{4}\right],\ \ h_1=\left[\frac{1}{2\sqrt{2}},-\frac{1}{2\sqrt{2}}\right]\ \ and\ \ h_2=\left[\frac{1}{4},-\frac{1}{2},\frac{1}{4}\right],
\end{equation}
are frequently applied examples that are identified as the two lowest-order ones in this class respectively. In order to construct two-dimensional (and other high-dimensional) examples out of $L^2(\mathbb R)$ wavelet frames, one may resort to the \emph{tensor product} construction, which specifies
\begin{equation}
\phi^{\otimes}(x,y)=\phi(x)\phi(y)
\end{equation}
as the refinable function, and
\begin{equation}
h^{\otimes}_{i,j}[k,l]=h_i[k]h_j[l]
\end{equation}
where $(i,j)\in\{1,\cdots,J\}^2$ as the corresponding filters. For more extensive elaboration on wavelets and wavelet frames in particular, we suggest \cite{notes} as a reference.

Before stating the construction of our model out of these mathematical techniques, we further introduce some notations. Conforming to the dilation convention, we denote the pixel width at resolution $n$ as $h_n:=2^{-n}$. A digital image signal $\mathbf u_n$ is modeled as a real-valued function defined on a subset of the discrete lattice $\mathbb O_n=h_n^{-1}\Omega\cap\mathbb Z^2$. Very often, in particular, we wish to analyze digital signals obtained from natural (analogue) ones, and for this purpose we consider the \emph{sampling operator} $\mathbf T_n$ which is defined by identity
\begin{equation}\label{sampling_operator}
\mathbf T_nu=2^n\left<u,\phi_{n,k}\right>,  
\end{equation}
where $u\in L^2(\Omega)$ and $k\in\mathbb K_n:=\{k\in\mathbb O_n:supp(\phi_{n,k})\subset\Omega\}$, the subset of $\mathbb O_n$ on which the pairing is well-defined. In this framework, we perform analysis regarding the properties of our discrete model, in particular, its asymptotic properties. It is worth emphasizing that although our expositions that follow always involve the resolution index $n$, it is however only for analytical convenience and heuristic purposes, and at any \emph{particular} resolution the model gains significance on its own.

In order to effectively characterize regularity as well as various singular phenomena of the image $\mathbf u_n$, we apply the ($\lambda_{n,j}$-weighted) \emph{discrete wavelet transform} $\mathbf W_n$ to this image and obtain its coefficients,
\begin{equation}\label{wavelet_coeff}
c_{n,j}[k]:=\lambda_{n,j}\cdot\mathbf W_n\mathbf u_n[j;k]:=\lambda_{n,j}\cdot(h_{n,j}*\mathbf u_n)[k],
\end{equation}
where $h_{n,j}$ stands for the convolution filter corresponding to channel $j$ and the weights $\lambda_{n,j}$ are adjustable model parameters. 
{ In fact, we can show that $c_{n,j}$ can be regarded as sampling on the derivates of the image, see \eqref{coeff_local}.} 
In order to attain more powerful and more robust time-frequency analysis around a given location $k\in\mathbb K_n$, we define its \emph{lattice neighbourhood} $\mathbb B_{n,k}$ as
\begin{equation}
\mathbb B_{n,k}=\left\{l\in\mathbb K_n:|h_nl-h_nk|\leq H_n\right\},
\end{equation}
where the radius $H_n$ can be adjusted for various practical purposes. In general, our argument should be applicable for any choice of $\{H_n\}_{n=1}^\infty$ that satisfies $H_n\to0$ and $h_n/H_n\to0$, but for the sake of simplicity, we specify $H_n=2^{-n/2}$ throughout the article. For later convenience, we also reserve a notation for the channel-wise summed coefficients as follows, 
\begin{equation}
c_n[k]:=|\mathbf W_n\mathbf u_n|[k]:=\left(\sum_{j=1}^J\mathbf W_n\mathbf u_n[j;k]^2\right)^{\frac{1}{2}}.
\end{equation}

And finally, we specify the transform $\hat{\mathbf W}_n$, the central object of our model design, as a truncated $\ell^r-$sum of the relevant coefficients, this time within the designated neighbourhood $\mathbb B_{n,k}$, literally,
\begin{equation}
\hat c_n[k]:=\hat{\mathbf W}_n\mathbf u_n[k]:=\min\left\{M_n,\left(\sum_{\mathbf l\in\mathbb B_{n,k}}\left|\mathbf W_n\mathbf u_n\right|[l]\ ^r\right)^{\frac{1}{r}}\right\},
\end{equation}
where the truncation bound $M_n\in(0,\infty]$. Here the parameter $r$ is designed to control the way in which the neighbouring coefficients are accumulated together. In particular, the scenario $r=1$ corresponds to plain averaging, and $r=\infty$, on the other hand, to the local-maximum operation.

The model that we consider in this article is eventually formulated as a minimization problem over the following energy, 
\begin{equation}\label{def_model}
\mathbf E_n(\mathbf u_n)=h_n^2\sum_{k\in\mathbb K_n}\left|\hat{\mathbf W}_n\mathbf u_n[k]\right|^p+h_n^2\sum_{k\in\mathbb O_n}((\mathbf A_n\mathbf u_n)[k]-\mathbf f_n[k])^2,
\end{equation}
in which the first term is a wavelet measurement of the image's local variations, summed globally; and the second term is just the ordinary $\ell^2$-fidelity that keeps prospective solutions reasonably close to the observation $\mathbf f_n$.

\subsection{A Mumford--Shah type model (continuum)}\label{continuum_model}

Continuum image processing also provides an inspiring perspective. Considering images simultaneously in the discrete and the continuum spaces illuminates subtle connections between the crucial operations involved in the respective models, which is advantageous for enhancing our understanding of model designing in general. Here we introduce a variation based model that incorporates a geometric measure count for singularities in continuum images, which is characteristic of the Mumford-Shah model and its derivatives. Later we show that it becomes exactly the asymptotic limit of our discrete wavelet-based model when resolution $n$ goes to infinity, as far as a number of reasonable assumptions are satisfied.

To give the precise definition of this (Mumford--Shah type) functional in question, we need to introduce the environment in which we will mainly work throughout this paper, namely, the space of \emph{special bounded variation (SBV) functions}, a useful subspace of the standard BV space. The property of SBV function space and its related variational problems are discussed in \cite{ambrosio2000functions}. At this moment we first recall the definition and several fundamental analytical properties of BV functions. Readers may refer to \cite{evans2015measure} for a more extensive introduction.

Let $\Omega$ denote an open domain in $\mathbb R^d$ (where $d=2$ in most of our relevant cases). A function $u\in L^1(\Omega)$ is called a \emph{function of bounded variation} if its \emph{total variation}
\begin{equation}
|u|_{BV(\Omega)}=\sup_{v\in C^\infty(\Omega;\mathbb R^2), \|v\|_{L^{\infty} \le 1}}\int_\Omega u(x)\ (\nabla\cdot v)(x)dx<\infty.
\end{equation}
The collection of all such functions is known to form a Banach space, denoted $BV(\Omega)$, where its norm is defined as $\|u\|_{BV(\Omega)}=\|u\|_{L^1(\Omega)}+|u|_{BV(\Omega)}.$

\begin{theorem}[{\cite[Definition 3.91]{ambrosio2000functions}}]
If the function $u\in BV(\Omega)$, its distributional gradient has the following (unique) standard decomposition,
\begin{equation}
\nabla u=\nabla^au+\nabla^su+Cu,
\end{equation}
where $\nabla^au$ is an absolutely continuous measure (with respect to the $d$-dimensional Lebesgue measure), $\nabla^su$ is a singular measure that is supported on the (approximation) jump sets (\cite[Definition 3.67]{ambrosio2000functions}, or \cite[Section 5.9]{evans2015measure}), and $Cu$ is a measure that is perpendicular to both of the former ones.
\end{theorem}

In image processing, the absolutely continuous measure $\nabla^a u$ quantifies mild gradient changes presenting in images, while $\nabla^su$ describes jump discontinuities (in other words, sharp edges) in them. The latter, in particular, can be expressed as a sum of point/curve supported mass in dimensions 1 and 2 respectively. Both of these two components encode important features that are considered criteria for successful restoration. Comparatively, the third part $Cu$ corresponds to fractal singularities in images, and it is thus often considered to be negligible in practice. For this purpose, one often resorts to the following modified definition for image modeling.
    
\begin{definition}[{\cite[Section 4.1]{ambrosio2000functions}}]
\label{def_sbv}
Assume $u\in BV(\Omega)$, and the standard decomposition of its distribution gradient $\nabla u$ is expressed as $\nabla u=\nabla^a u+\nabla^s u+Cu.$ This function $u$ is called a special function of bounded variation, denoted $u\in SBV(\Omega)$, if $Cu=0$. Here the norm of SBV function is induced from the original BV norm.
\end{definition}

Most properties of the SBV function space, such as compactness, lower semi-continuity, etc., were discussed in \cite[Section 4.1]{ambrosio2000functions}. In the same reference, the existence of minimizers to relevant variational problems was extensively discussed, including that to the Mumford--Shah functional considered in this paper. Some crucial facts, in particular, are summarized into the following classical proposition on what is usually called \emph{the fine structure of SBV functions}.

\begin{proposition}[{\cite[Section 5.9]{evans2015measure}}]\label{fine_property}
    Let $u\in BV(\Omega)$, it follows that the singular gradient $\nabla^s u$ is supported on $\Gamma_u$, a countable union of $C^1$ rectifiable curves, that is, $\Gamma_u\subset\bigcup_{i=1}^\infty\Gamma^{(i)}$, where each $\Gamma^{(i)}$ denotes a $C^1$-curve. 
    If $u \in SBV(\Omega)$, then for each $x \in \Gamma_u$, with normal direction $\nu$, there exists $u^+_x$, $u^-_x$ (denoted as $u^+$, $u^-$ respectively if the context is clear), such that for $\mathcal H^1-$a.e. point $x$, 
    \begin{equation}
        \label{jump_local}
        u(\varepsilon y + x) \to u^+ \chi_{y \cdot \nu > 0} + u^- \chi_{y \cdot \nu < 0}, 
    \end{equation}
    in $L^2(dy; |y| \le 1)$ sense, as $\varepsilon \to 0$. 
    Similarly, for each $x \not \in \Gamma_u$, then there exists $\overline{u}$ such that for a.e. $\mathcal H^1$ point $x$, 
    \begin{equation}
        \label{smooth_local}
        u(\varepsilon y + x) \to \overline{u}.
    \end{equation}    
\end{proposition}
{
The proposition says that if we zoom in the function at a point on the jump set, then we can ultimately obtain a jump of some hyperplane. On the other hand, if the point is not on the jump set, the proposition will reproduce the famous Lebesgue's theorem.}

The aforementioned continuum limit of our discrete model is formulated by minimizing the following variational energy functional. 
\begin{equation}\label{energy_continuum}
E(u)=\left(\int_{\Omega\setminus\Gamma_u}|\nabla u|+\alpha\cdot\mathcal H^1(\Gamma^*_u)\right)+\int_{\Omega}|Au-f|^2,
\end{equation}
where
\begin{equation}
\Gamma^*_u=\{x\in\Gamma_u:\rho_u(x):= u^+_x - u^-_x\neq 0\}
\end{equation}
denotes the essential singularity of $u$. In particular, by Proposition \ref{fine_property}, $\mathcal H^1(\Gamma_u\setminus\Gamma^*_u)=0$. This functional is largely akin to the classical Mumford-Shah energy \cite{mumford1985boundary}, but rather than taking infimum with respect to all possible exception sets, it directly specifies as exception the function's jump singularity in the $SBV(\Omega)$ sense, $\Gamma^*_u$, so as to calculate its energy separately. 

Note that $\mathcal H^1(\Gamma_u^*) = \mathcal H^1(\Gamma_u)$, so the existence of the minimizer is just the same as the classical Mumford--Shah model formulated in the SBV function space, see the discussions in \cite[Chap. 7]{ambrosio2000functions}. We reserve the notation $\Gamma_u^*$ for later technical convenience.

\subsection{Model analysis: the main theorem}

Our main result is about the asymptotic correspondence between the discrete model \eqref{def_model} and the continuum model \eqref{energy_continuum} mentioned above. In order to state it precisely, we need to detail some technical assumptions regarding the formation of the wavelets applicable. 

\begin{assumption}\label{assumption_R}
    The assumptions on model \eqref{def_model} consist of the followings:
    
    (a) The wavelet collection $\Psi=\{\psi_j\}_{j=1}^J$ applied includes at least two elements, $\psi_1$ and $\psi_2$, which have vanishing moments $(1,0)$ and $(0,1)$ respectively. Their corresponding compactly supported anti-derivatives, $\varphi_j=(\partial_j)^{-1}\psi_j$ are either non-negative or non-positive.
    
    (b) There exists $C_{wav}>0$ such that $\emph{supp}(\psi_{j})\subset B(0,C_{wav})$.
    
    (c) There exists $C_2>0$ such that
    
    \begin{equation}\label{robust_supp}
    \Delta:=\inf_{\nu\in\mathbb S^1,a\in[-2,1]^2}\left(\sum_{j=1,2}\left(\int_{\mathbb R^2}\chi_{x\cdot\nu\geq 0}(x)\psi_j(x-a)dx\right)^2\right)^{\frac{1}{2}}\geq C_2.
    \end{equation}
    \end{assumption}

In general, we notice that in the above statement, condition \emph{(a)} requires that local (first order) derivatives of a target function can be sampled by wavelets of relevant vanishing moments. This is fundamentally integration by part, and interested readers may refer to \cite{tv-wf,wf-pde} for further explanations and technical applications. The other two conditions, \emph{(b) and (c)}, are imposed concerning the derivative sampler $\varphi_j$'s support, respectively, its \emph{robust support}. More specifically, if we denote $\mathbb D_1=\{(\cos\theta,\sin\theta)^\top:\theta\in[-\pi/3,\pi/3]\cup[2\pi/3,4\pi/3]\}$, $\mathbb D_2=\{(\cos\theta,\sin\theta)^\top:\theta\in[\pi/6,5\pi/6]\cup[7\pi/6,11\pi/6]\}$ and $I_j(\nu,a)=\int_{\mathbb R^2}\chi_{x\cdot\nu\geq0}(x)\psi_j(x-a)dx$ as appearing in \eqref{robust_supp}, we have
\begin{eqnarray}
|I_j|=|\left<\delta_\nu,\varphi_j(\cdot-a)\right>\nu_j|
\end{eqnarray}
where $\delta_\nu=\mathcal H^1\lfloor_{x\cdot\nu=0}$, and in turn,
\begin{eqnarray}
    \Delta&=&\inf_{\nu\in\mathbb S^1,a\in[-2,1]^2}\sqrt{I_1^2+I_2^2}\\
    &=&\min_{j=1,2}\left\{\inf_{\nu\in\mathbb D_j,a\in[-2,1]^2}\sqrt{I_1^2+I_2^2}\right\}\\
    &\geq&\min_{j=1,2}\left\{\inf_{\nu\in\mathbb D_j,a\in[-2,1]^2}|I_j|\right\}\\
    &\geq&\frac{1}{2}\cdot\min_{j=1,2}\left\{\inf_{\nu\in\mathbb D_j,a\in[-2,1]^2}|\left<\delta_\nu,\varphi_j\right>|\right\}=:\Delta'.
\end{eqnarray}
In some cases, it would be more convenient to verify the stronger (sufficient) condition $\Delta'\geq C_2>0$ instead. In particular, we note that for $B$-splines of order $k\geq6$, this condition is valid, together, indeed, with others specified in the assumption.

Also, we need the consistency (\emph{i.e.} asymptotic commutativity) of the sampling operator with respect to the processing operator family.
    
\begin{assumption}\label{assumption_F}
    The processing operator $A$ and its discrete counterparts $\mathbf A_n$ together satisfy 
\begin{equation}
\lim_{n\to\infty}\|\mathbf T_n(Af)-\mathbf A_n\mathbf T_nf\|_{L^2(\Omega)}=0,
\end{equation}
where $\mathbf T_n$ denotes the sampling operator defined as in \eqref{sampling_operator}.
\end{assumption}
In fact, many frequently encountered discrete/continual operators have a well-behaved counterpart in the opposite category, such as \emph{convolution} in deblurring tasks and \emph{(domain) restriction} in inpainting, with which the assumption can be verified~\cite{tv-wf}.

With these necessary technical assumptions, we attempt to identify the asymptotic (continuum) limit of our wavelet based two-scale model. Since the argument $\mathbf u_n$ of the model $\mathbf E_n$ has variable resolution $2^n$ (cf. \eqref{def_model}), we need to secure a ground on which to facilitate comparison. In this case, we again treat the discrete image in question as a sample, literally $\mathbf u_n=\mathbf T_nu_n$, where $u_n\in SBV(\Omega)$. Formally, $\mathbf E_n(\mathbf u_n)=(\mathbf E_n\circ\mathbf T_n)(u_n)$, where $\mathbf E_n\circ\mathbf T_n$ becomes a functional defined on the continuum space $SBV(\Omega)$. We analyze the sequence $\{\mathbf E_n\circ\mathbf T_n\}_{n=1}^\infty$ instead of $\{\mathbf E_n\}_{n=1}^\infty$, and denote the composition $\mathbf E_n\circ\mathbf T_n$ by plain font $E_n$ for convenience. 

The model limit would in turn be identified in the context of \emph{$\Gamma$-convergence}, regarding which we recall that: 

\begin{definition}
A sequence of functionals $\{F_n\}_{n=1}^\infty$ is said to $\Gamma$-converge to a limit functional $F$ (denoted $F_n\xrightarrow{\Gamma} F$) in a Banach space $\mathcal B$ if the following two conditions are satisfied together:

(i) Universal: $\liminf_{n\to\infty}F_n(u_n)\geq F(u)$ for \emph{any} $u_n\to u\in \mathcal B$;

(ii) Existence: $\limsup_{n\to\infty}F_n(u_n^*)\leq F(u)$ for \emph{at least one} sequence $u_n^*\to\ u\in \mathcal B$.
\end{definition}

Finally, we specify the relevant model parameters and attain our desired asymptotic analysis as follows:

\begin{theorem}\label{main_theorem}
Fix a quasi-affine wavelet frame system $X^q(\Psi)$ that satisfies Assumption \ref{assumption_R}, and suppose the processing operators $\mathbf A_n$, $A$ are given in accordance with Assumption \ref{assumption_F} with respect to sampling $\mathbf T_n$. In this circumstance, if we specify
$\lambda_{n,j}=-2^{n-2}\left(\int_{\mathbb R^2}\varphi_j\right)^{-1}$ 

and
$M_n=\frac{\alpha}{2H_n}$
in defining the model energy (cf. \eqref{def_model}), 

then there will be convergence $E_n\xrightarrow{\Gamma} E$ in $SBV(\Omega)$ as $ n \to \infty.$
\end{theorem}

Theorem \ref{main_theorem} is to be established as a consequence of the following propositions, which characterize the asymptotics of regularity and fidelity energy functionals separately. The fidelity term has been separately considered in the previous work \cite{tv-wf}, in which they presented a result that was actually stronger than, namely, gamma convergence (which is not additive in general).

\begin{proposition}[\cite{tv-wf}]\label{limfid}
    Fix a function $f\in L^2(\Omega)$. For any variable $u\in L^2(\Omega)$, let its discrete fidelity functional (with respect to $f$) at resolution $n$ be defined as
    \begin{equation}
    F_n(u):=\|\mathbf A_n\mathbf T_nu-\mathbf T_nf\|_{\ell^2(\mathbb O_n)}^2,
    \end{equation}
    and that at continuum resolution as
    \begin{equation}
    F(u):=\|Au-f\|_{L^2(\Omega)}^2.
    \end{equation}
    If the operators $\mathbf A_n$ and $A$ presenting in these terms satisfy Assumption \ref{assumption_F} with respect to $\mathbf T_n$, then:
    
    (i) There is point-wise convergence
    \begin{equation}
    \lim_{n\to\infty}F_n(u)=F(u);
    \end{equation}
    
    (ii) The functional set $\{F_n\}_{n=1}^\infty$ is equi-continuous on $L^2(\Omega)$.
\end{proposition}
In this way, we are left only to analyze the regularity term of the model. We conduct the liminf part (Theorem~\ref{liminf}) and the limsup part (Theorem~\ref{limsup}) of our analysis separately. 

\begin{proposition}\label{liminf}
For any variable function $u\in SBV(\Omega)$, we separately denote its regularity functional at resolution $n$ as
\begin{equation}\label{definition_R_n}
R_n(u):=h_n^2\sum_{k\in\mathbb K_n}\left|\hat{\mathbf W}_n\mathbf T_nu[k]\right|
\end{equation}
and, respectively, that at continuum (infinite) resolution as
\begin{equation}\label{definition_R}
R(u):=\int_{\Omega\setminus\Gamma_u}|\nabla u|+\alpha\cdot\mathcal H^1(\Gamma^*_u).
\end{equation}
If the parameters involved are specified such that Assumption \ref{assumption_R} is satisfied, we have
\begin{equation}
\liminf_{n\to\infty}R_n(u_n)\geq R(u)
\end{equation}
for any sequence $\{u_n\}_{n=1}^\infty\subset SBV(\Omega)$ such that $\|u_n-u\|_{SBV(\Omega)}\to0$.
\end{proposition}

\begin{proposition}\label{limsup}
Define functionals $R_n$, $R$ by \eqref{definition_R_n} and \eqref{definition_R} respectively.

There exists a sequence $\{u^*_n\}_{n=1}^\infty\subset SBV(\Omega)$ such that $\|u^*_n-u\|_{SBV(\Omega)}\to0$ and 
\begin{equation}
\limsup_{n\to\infty}R_n(u^*_n)\leq R(u).
\end{equation}
\end{proposition}

The establishment of Proposition \ref{liminf} and Proposition \ref{limsup} necessitates a comprehensive understanding of the image structure (particularly its singularity formation) and therefore we organize all relevant technical details separately into the two sections that follow. Here we first show that our major result, Theorem \ref{main_theorem}, can be implied by them together.

\begin{proof}[Proof of Theorem \ref{main_theorem}]
Assume $u\in SBV(\Omega)$. By embedding theorem, we know that $u\in L^2(\Omega)$. It follows from Proposition \ref{limfid} that given any $\epsilon>0$, there is an $N_1\in\mathbb Z^+$ such that $|F_n(u)-F(u)|<\epsilon$ for any $n\geq N_1$. Secondly, for this $\epsilon$, there is a $\delta>0$ such that $|F_n(u')-F_n(u)|<\epsilon$ for any $\|u'-u\|_{SBV(\Omega)}<\delta$ and any $n\geq1$ by equi-continuity.

For arbitrary $u_n\to u$ in $SBV(\Omega)$, there exists an $N_2\in\mathbb Z^+$ such that $\|u_n-u\|_{SBV(\Omega)}<\delta$ for any $n\geq N_2$. Consequently, if $n\geq\max\{N_1,N_2\}$, we have
\begin{equation}
|F_n(u_n)-F(u)|\leq|F_n(u_n)-F_n(u)|+|F_n(u)-F(u)|\leq2\epsilon,
\end{equation}
that is, $\lim_{n\to\infty}F_n(u_n)=F(u)$. Together with Proposition \ref{liminf}, we have $\liminf_{n\to\infty}E_n(u_n)\geq E(u)$, noting that $E_n(u_n)=R_n(u_n)+F_n(u_n)$ and $E(u)=R(u)+F(u)$ respectively.

By Proposition \ref{limsup}, there exists a sequence $u_n^*\to u$ in $SBV(\Omega)$ such that $\limsup_{n\to\infty}R_n(u_n^*)\leq R(u)$. This sequence, in particular, also has $\lim_{n\to\infty}F_n(u_n^*)=F(u)$, and consequently we conclude that $\limsup_{n\to\infty}E_n(u_n^*)\leq E(u)$.

\end{proof}

\subsection{Toy models and heuristics} 
{We begin with some straightforward calculations leading to the master key identity

\begin{equation}
    \label{coeff_local}
    c_{n,j}[u;k] = \left<\partial_ju,I_j^{-1}\varphi^1_{j,n-1,k}\right>,
\end{equation}
where $j=1,2$ in particular. This identifies wavelet coefficients (cf. \eqref{wavelet_coeff}) as derivative samplings \cite{tv-wf,wf-pde}, and will be applied repeatedly throughout our argument. 

By definition, $\mathbf T_nu=2^n\left<u,\phi_{n,k}\right>$, therefore it follows
\begin{eqnarray}
c_{n,j}[u;k]&=&\lambda_{n,j}\cdot\sum_{l\in\mathbb Z^2}h_j[k-l]\cdot2^n\left<u,\phi_{n,l}\right>\\
&=&\lambda_{n,j}\cdot2^n\left<u,\sum_{l\in\mathbb Z^2}h_j[k-l]\phi_{n,l}\right>\\
&=&\lambda_{n,j}\cdot2^n\left<u,\psi_{j,n-1,k}\right>\left(=-I_j^{-1}2^{n-1}\left<u,\psi^1_{j,n-1,k}\right>\right)\\
&=&\lambda_{n,j}\cdot2\left<u,\partial_j\varphi_{j,n-1,k}\right>\\
&=&\lambda_{n,j}\cdot(-2)\left<\partial_ju,\varphi_{j,n-1,k}\right>\\
&=&\left<\partial_ju,I_j^{-1}\varphi^1_{j,n-1,k}\right>,
\end{eqnarray}
where we have recalled that $\lambda_{n,j}:=-2^{n-2}I_j^{-1}$. 
Here it is worth noting that the derivatives $\partial_j u$ involved are generally distributional in nature, consequently, the second last identity is validated as a dual pairing (dualization) rather than as an inner product.  
}

In order to quantify the difference between functions defined on a continuum domain and their discrete samplings, we also need the following approximation lemma. 
This lemma indicates that wavelet coefficients do recover the gradient of the function asymptotically in the sense of Lebesgue integration, in particular, how discrete samplings are related to their continuum counterpart when resolution goes to infinity.

\begin{lemma}[\cite{tv-wf}]
\label{approximation_theorem}
    Assume $f\in L^1(\Omega)$. If $\bar\varphi_{n,k}(x)=2^{2n}\varphi(2^nx-k)$ and $\chi_{n,k}(x)=\chi(2^nx-k)$, where $\chi$ satisfies partition of unity and $supp(\chi)\subset supp(\bar\varphi)$, then we have
    \begin{equation}
    \lim_{n\to\infty}\left\|\sum_{k\in\mathbb O_n,supp(\bar\varphi_{n,k})\subset\bar\Omega}\left<f,\bar\varphi_{n,k}\right>\chi_{n,k}-f\right\|_{L^1(\Omega)}=0.
    \end{equation}
    
\end{lemma}

As a basis for understanding further our core concept of (wavelet frame based) singularity modeling, we first consider with the help of the above lemma an ideal case, in which the argument function $u$ is simply regular on the whole domain $\Omega$. 
\begin{proposition}
    \label{prop:continuous}
    Suppose $u\in C^1(\Omega)$ and its gradient $\nabla u$ is uniformly continuous on $\Omega$,
    
    we have point-wise convergence $\lim_{n\to\infty}E_n(u)=E(u)$. 
\end{proposition}
\begin{proof}
Since the fidelity term's convergence has been established already in \Cref{limfid}, we only need to prove $\lim_{n\to\infty}R_n(u)=R(u)$. The key observation here is that the (neighbourhood) maximal operator does not bring about significantly enlarged coefficients if $u$ itself is sufficiently regular. 

By assumption, the gradient $\nabla u$ is uniformly continuous, on $\Omega$, that is, for any $\epsilon>0$ there exists a $\delta>0$ such that $|x-y|<\delta$ ($x,y\in\Omega$) implies $|\nabla u(x)-\nabla u(y)|<\epsilon$. For all sufficiently large $n$, that is, such that $H_n<\delta$, we have $|c_n[u;l]-c_n[u;k]|<\epsilon$ for any $k\in\mathbb K_n$ and any $l\in\mathbb B_{n,k}$. Note that for sufficiently large $n$, we always have $c_n[u;k]<M_n$ as $u$ is continuous, and $\hat c_n[u;k]=\max_{l\in\mathbb B_{n,k}}c_n[u;k]$. Consequently, $|\hat c_n[u;k]-c_n[u;k]|<\epsilon$ for any $k\in\mathbb K_n$, and by summing them together,
\begin{equation}
\left|h_n^2\sum_{k\in\mathbb K_n}\hat c_n[u;k]-h_n^2\sum_{k\in\mathbb K_n}c_n[u;k]\right|<\epsilon.
\end{equation}
On the other hand, by Lemma \ref{approximation_theorem} and the fact $\Gamma_u=\emptyset$, we have
\begin{equation}
\lim_{n\to\infty}\left(h_n^2\sum_{k\in\mathbb K_n}c_n[u;k]\right)=\int_\Omega|\nabla u|=R(u).
\end{equation}
Put them together and we obtain our desired convergence. 

\end{proof}

We complete this section with the following estimation concerning the regularity energy of a function $u$ that is static with respect to $n$ but is allowed to possess some well-formed singularity. In Section \ref{proof_liminf}, we will eventually analyze problems involving varying $u_n$ with general singularity by transforming them into this prototype.

\begin{lemma}\label{limsup_regular}
    If the function $u \in SBV(\Omega)$ has a boundary extension toward $\partial \Omega$ that belongs to $C^1(\bar{\Omega}\setminus\tilde\Gamma)$, and its gradient $\nabla u$ is bounded on $\Omega\setminus\tilde\Gamma$ (where $\tilde\Gamma$ is a finite union of closed curve segments without self- and mutual intersections), we have
    \begin{equation}
    \limsup_{n\to\infty}R_n(u)\leq\int_\Omega|\nabla^au|+\alpha\cdot\mathcal H^1(\tilde\Gamma).
    \end{equation}
    \end{lemma}
\begin{proof}
Denote $B=\sup_{x\in\Omega\setminus\tilde\Gamma}|\nabla u(x)|<\infty$. Let $\epsilon>0$, and $\tilde\Gamma^\epsilon$, $\tilde\Gamma_n$ denote tubular neighbourhoods of $\tilde\Gamma$ of width $\epsilon$ and $H_n+C_{wav}\cdot2^{-n}$ respectively, where $C_{wav}$ is defined in Assumption \ref{assumption_R}. In particular, $supp(\psi_{j,n,k})\subset B(2^{-n}k;C_{wav}\cdot2^{-n})$, and particularly, for any $k\in\mathbb K_n$ such that $2^{-n}k\in\tilde\Gamma_n$, we have $supp(\psi_{j,n,k})\cap\tilde\Gamma=\emptyset$ thus $c_{n,j}[u;k]\leq B$ and $c_n[u;k]\leq\sqrt2 B$.

Since $\nabla u=\nabla^au$ on $\Omega\setminus\tilde\Gamma$, coefficients 

$c_{n,j}[u;k]=c^a_{n,j}[u;k]$
for $k\in\mathbb K_n\setminus2^n\tilde\Gamma_n$, where

\begin{equation}
c^a_{n,j}[u;k]:=-2\lambda_{n,j}\left<\partial^a_ju,\varphi_{j,n-1,k}\right>
\end{equation}
is defined as an analogue of wavelet coefficients (one that excludes the singularity effects), which will be frequently applied later in our analysis. Moreover, $\hat c_n[u;k]\leq M_n=\frac{\alpha}{2H_n}$ for any $k\in\mathbb K_n$. Since $\nabla u$ (by extension) is continuous on $\bar\Omega\setminus\tilde\Gamma$, it follows that on $\bar\Omega\setminus\tilde\Gamma^\epsilon$, for the above specified $\epsilon>0$ there exists $\delta>0$ such that $|\nabla u(x)-\nabla u(x')|<\epsilon$ whenever $x,x'\in\Omega\setminus\tilde\Gamma^\epsilon$ and $|x-x'|\leq\delta$. In what follows, we assume $n\geq N$ where $N$ is taken such that $2H_N<\delta$ is satisfied, then for any $k\in\mathbb K_n\setminus2^n\tilde\Gamma_n$, we have $\left|\max_{l\in\mathbb B_{n,k}} c^a_n[u;l]-c^a_n[u;k]\right|<\epsilon$ (where $c^a_n[u;l]:=(\sum_{j}c^a_{n,j}[u;l]^2)^{1/2}$) and therefore it immediately follows that \begin{equation}\label{eq:hateps}\hat c_n[u;k]<c^a_n[u;k]+\epsilon.\end{equation}
    
We split the target quantity into three parts,

\begin{equation}
\begin{split}
\limsup_{n\to\infty}R_n(u)
&\leq\limsup_{n\to\infty}\left(2^{-2n}\sum_{k\in\mathbb K_n\setminus2^n\tilde\Gamma^\epsilon}\hat c_n[u;k]\right)\\
    &\ \ \ \ +\limsup_{n\to\infty}\left(2^{-2n}\sum_{k\in \mathbb K_n\cap(2^n\tilde\Gamma^\epsilon\setminus2^n\tilde\Gamma_n)}\hat c_n[u;k]\right)
    \\ & +\limsup_{n\to\infty}\left(2^{-2n}\sum_{k\in\mathbb K_n\cap2^n\tilde\Gamma_n}\hat c_n[u;k]\right),
\end{split}
\end{equation}
and estimate them separately. We first of all note that 
\begin{equation}
\begin{split}
    \limsup_{n\to\infty}\left(2^{-2n}\sum_{k\in\mathbb K_n\setminus2^n\tilde\Gamma^\epsilon}\hat c_n[u;k]\right) 
    &\le \limsup_{n\to\infty}\left(2^{-2n}\sum_{k\in\mathbb K_n\setminus2^n\tilde\Gamma^\epsilon}(c^a_n[u;k]+\epsilon)\right)\\
    &\leq \limsup_{n\to\infty}\left(2^{-2n}\sum_{k\in\mathbb K_n}c^a_n[u;k]\right)+\epsilon,
\end{split}
\end{equation}

which in turn is equal to $\int_{\Omega}|\nabla^au|+\epsilon$ by Lemma \ref{approximation_theorem}.
Secondly, applying gradient bound $c_n[u;k]\leq\sqrt2B$ at the off-singularity domain $\mathbb K_n \cap (2^n\tilde\Gamma^\epsilon\setminus2^n\tilde\Gamma_n)$, we have 
\begin{equation}
\begin{split}
\limsup_{n\to\infty}\left(2^{-2n}\sum_{k\in \mathbb K_n\cap(2^n\tilde\Gamma^\epsilon\setminus2^n\tilde\Gamma_n)}\hat c_n[u;k]\right) & \le \limsup_{n\to\infty}\left(2^{-2n}\sum_{k\in\mathbb K_n\cap(2^n\tilde\Gamma^\epsilon\setminus2^n\tilde\Gamma_n)}\sqrt2B\right) \\ &\le \sqrt2B\cdot\mathcal H^2(\tilde\Gamma^\epsilon).
\end{split}
\end{equation}
Finally, at the near-singularity domain $\mathbb K_n\cap2^n\tilde\Gamma_n$, we have 
\begin{equation}
    \limsup_{n\to\infty}\left(2^{-2n}\sum_{k\in\mathbb K_n\cap2^n\tilde\Gamma_n}\hat c_n[u;k]\right) \le \limsup_{n\to\infty}\left(2^{-2n}\sum_{k\in\mathbb K_n\cap2^n\tilde\Gamma_n}\frac{\alpha}{2H_n}\right) \le \alpha\cdot\mathcal H^1(\tilde\Gamma).
\end{equation}
Combining the above arguments, we totally have 
\begin{equation}
    \limsup_{n\to\infty}R_n(u)\le \int_{\Omega}|\nabla^au|+\alpha\cdot\mathcal H^1(\tilde\Gamma)+\left(\epsilon+\sqrt2B\cdot\mathcal H^2(\tilde\Gamma^\epsilon)\right).
\end{equation}

    Since the parameter $\epsilon>0$ can be taken arbitrarily and $\lim_{\epsilon\to0}\mathcal H^2(\tilde\Gamma^\epsilon)=0$, our proof is concluded immediately.
\end{proof}

\section{{Asymptotic analysis of the regularity term}}\label{proof_liminf}

{This section fulfills two aims.} In what follows, we first present a series of results that eventually lead to \Cref{liminf}, providing the asymptotic lower-bound $E(u)$ of the discrete regularity energy. Finally, in Subsection \ref{section_tightness} we establish Proposition \ref{limsup}, which validates $E(u)$ as a $\Gamma$-limit by virtue of its tightness.

Frequently revisited notations in this section are summarized for reference as follows: 

\begin{itemize}
    \item[-] Wavelet coefficients $c_n[u;k]=|\mathbf W_n\mathbf T_nu|[k]$, where $\mathbf W_n$ and $\mathbf T_n$ denote respectively the discrete wavelet transform and the sampling operator that are defined in Section~\ref{model_discrete}
    .
    \item[-] Jump (magnitude) $\rho(x) = |u^+(x) - u^-(x)|$ as is characterized by Proposition \ref{fine_property}.
    \item[-] The neighbourhood radius $H_n = 2^{-n/2}$ and the threshold $M_n = \alpha / (2H_n)$.
\end{itemize}

Our key idea here is to construct a suitable truncation of $u$'s jump singularity, which eliminates all small (but potentially ubiquitously distributed) jumps and preserves only a certain chunk of it. In turn, we estimate the asymptotic lower-bound of its energy that occurs close to, respectively, far away from this curve-supported singularity. 

As our analysis involves varying (converging) sequences in $SBV(\Omega)$, we need a certain sort of stability control over wavelet coefficients. In particular, we have the following proposition.

\begin{lemma}[Stability of plain coefficients]
\label{coeff_sum}
    
    For any $u',u\in SBV(\Omega)$, we have difference measurement
    \begin{equation}
    2^{-2n}\sum_{k\in\mathbb K_n}|c_n[u';k]-c_n[u;k]|\leq C\|u'-u\|_{SBV(\Omega)},
    \end{equation}
    where $C>0$ denotes a constant (that depends on the wavelets chosen).
    \end{lemma}

\begin{proof}

Denote $\Omega_n:=\{x\in\Omega:\lfloor2^nx\rfloor\in\mathbb K_n\}$. For $j=1,2$, we both have
\begin{eqnarray*}
&&2^{-2n}\sum_{k\in\mathbb K_n}|c_{n,j}[u';k]-c_{n,j}[u;k]|\\
&=&\int_{\Omega_n}|c_{n,j}[u';\lfloor2^nx\rfloor]-c_{n,j}[u,\lfloor2^nx\rfloor]|dx\\
&=&|I_j|^{-1}\int_{\Omega_n}\left|\int_{supp(\varphi_j)}\left(\partial_ju'(2^{-n}y+x)-\partial_ju(2^{-n}y+x)\right)\varphi_j(x-(\lfloor2^nx\rfloor-2^nx))dy\right|dx\\
&\leq&|I_j|^{-1}\|\varphi_j\|_\infty\int_{\Omega_n}\int_{supp(\varphi_j)}\left|\partial_ju'(2^{-n}y+x)-\partial_ju(2^{-n}y+x)\right|dydx\\
&=&|I_j|^{-1}\|\varphi_j\|_\infty|supp(\varphi_j)|\cdot\|u'-u\|_{SBV(\Omega)}.
\end{eqnarray*}
And consequently,
\begin{eqnarray*}
&&2^{-2n}\sum_{k\in\mathbb K_n}|c_n[u';k]-c_n[u;k]|\\
&\leq&2^{-2n}\sum_{j=1,2}\sum_{k\in\mathbb K_n}|c_{n,j}[u';k]-c_{n,j}[u;k]|\\
&\leq&C\cdot\|u'-u\|_{SBV(\Omega)},
\end{eqnarray*}
where we define $C:=\sum_{j=1,2}|I_j|^{-1}\|\varphi_j\|_\infty|supp(\varphi_j)|>0$.
\end{proof}

\subsection{Finite approximation of the singularity}
By Proposition~\ref{fine_property}, a function of bounded variation (in particular, an SBV function) has a countably $C^1-$rectifiable singularity set on its domain $\Omega$. Namely, the singularity of $u\in SBV(\Omega)$ is included in a set of the form $\Gamma=\bigcup_{i=1}^\infty\Gamma^{(i)}$, where each $\Gamma^{(i)}=\operatorname{Im}(\gamma^{(i)})$ for a $C^1-$bijection $\gamma^{(i)}:\mathbb R\to\Omega$.

\begin{lemma}[Simple curve union approximation of singularity]
\label{curve_truncation}
Assume $u\in SBV(\Omega)$, and denote its jump singularity set as $\Gamma:=\Gamma(u)$. For any prescribed constant $\epsilon>0$, there exists an \emph{$\epsilon$-singularity set}, denoted $\Gamma_\epsilon\subset\Gamma$, which satisfies the following approximation properties:

\begin{itemize}
    \item[(i)] The set $\Gamma_\epsilon$ is a {simple union of curves} , which means it can be expressed as
\begin{equation}\label{finite_curve}
\Gamma_\epsilon=\bigcup_{i=1}^N\gamma^{(i)}([a_i,b_i])
\end{equation}
for some finite number $N\in\mathbb Z^+$ and $a_i,b_i\in\mathbb R$ for each $i\leq N$, where each two of these $N$ component have disjoint tubular neighbourhoods.

\item[(ii)] The intersection $\Gamma_{\epsilon}\cap\Gamma^*$ is sufficiently dominant in $\Gamma^*$ by measure. Specifically,
\begin{equation}\label{coverage}
\mathcal H^1\left(\Gamma^*\cap\Gamma_\epsilon\right)>\left\{
\begin{matrix}
(1-\epsilon)\cdot\mathcal H^1(\Gamma^*) & \text{if\ \ } \mathcal H^1(\Gamma^*)<\infty \\
\epsilon^{-1} & \text{if\ \ } \mathcal H^1(\Gamma^*)=\infty,
\end{matrix}
\right.
\end{equation}
where we recall that $\Gamma^* = \{x \in \Gamma : \rho(x) = 0\}$, the non-trivial part of the jump singularity set. 

\item[(iii)] The singular gradient $\nabla^su$ is sufficiently concentrated on $\Gamma_{\epsilon}
$ as a measure. In particular, its variation residue satisfies the bound
\begin{equation}\label{residue}
\int_{\Omega\setminus\Gamma_\epsilon}|\nabla^su|<\epsilon.
\end{equation}
\end{itemize}
\end{lemma}
\begin{proof}
Without loss of generality, we assume the (rectified) singularity in question is curve-length parametrized, that is
\begin{equation}
\Gamma=\bigcup_{i=1}^\infty\Gamma^{(i)}=\bigcup_{i=1}^\infty\gamma^{(i)}([0,b_i]),
\end{equation}
where the $C^1$-mappings $\gamma^{(i)}$ have $|\dot\gamma^{(i)}|=1$ at all relevant points. Here each $\dot\gamma^{(i)}$, consequently, is continuous on $[0,b_i]$ and therefore is uniformly continuous on this interval, in particular, there exists an $\eta_i>0$ such that any $t,t'\in[0,b_i]$ satisfying $|t-t'|<\eta_i$ have $|\dot\gamma(t)-\dot\gamma(t')|<\frac{\pi}{4}$. It is easy to verify that any curve segment of the form $\gamma^{(i)}([a,a+\eta_i])$, where $[a,a+\eta_i]\subset[0,b_i]$, has no self-intersections. For simplicity of notations, in what follows we shall assume without loss of generality that this property has already been guaranteed for each of the curve segments $\gamma^{(i)}([0,b_i])$.

Then we iteratively resolve the mutual intersections in the following way. For each $i\geq2$, note first that $K_i:=\left(\gamma^{(i)}\right)^{-1}\left(\left(\bigcup_{j=1}^i\Gamma^{(j)}\right)\cap\Gamma^{(i)}\right)$ is a closed set in $[0,b_i]$, consequently $(0,b_i)\setminus K_i=\bigcup_{l=1}^\infty(\alpha_{i,l},\beta_{i,l})$ for some $\alpha_{i,l},\beta_{i,l}$, and we shall replace the curve segment $\Gamma^{(i)}$ by $\bigcup_{l=1}^\infty\gamma^{(i)}(\alpha_{i,l},\beta_{i,l})$.

\setlength{\unitlength}{1cm}
\begin{picture}(0,2)
    \put(0,0){
        \put(1.5,0){\line(0,1){1}}
        \put(0,0.5){\line(1,0){3}} 
        \put(0,0.49){\line(1,0){3}}
        \put(0,0.51){\line(1,0){3}}
        \put(3.2,0.5){$\Gamma^{(i)}$}
        \put(1.6,0.8){$\Gamma^{(j)}$}
    }
    \put(6,0.4){$\Rightarrow$}
    \put(8,0){
        \put(1.5,0){\line(0,1){1}} 
        \put(0,0.5){\line(1,0){3}} 
        \put(0,0.49){\line(1,0){1.25}}
        \put(0,0.51){\line(1,0){1.25}}
        \put(1.75,0.49){\line(1,0){1.25}}
        \put(1.75,0.51){\line(1,0){1.25}}
        \put(3.2,0.5){$\Gamma^{(i)}$}
        \put(1.6,0.8){$\Gamma^{(j)}$}
        \put(1.38,0.4){$)$}
        \put(1.45,0.4){$($}
    }
\end{picture}

Finally, for the given $\epsilon>0$, we first find some finite numbers $I$, $L_i$ such that the object $\bigcup_{i=1}^I\bigcup_{l=1}^{L_i}\gamma^{(i)}([\alpha_{i,l},\beta_{i,l}])$ satisfies all the above requirements concerning residues but with $\epsilon/2$ instead of $\epsilon$, and then we properly specify some $\epsilon_{i,l}>0$ such that the closed curve segments of the form $\gamma^{(i)}([\alpha_{i,l}+\epsilon_{i,l},\beta_{i,l}-\epsilon_{i,l}])$ have mutually disjoint tubular neighbourhoods (by tubular neighborhood theorem) and all the relevant residue loss are controlled by another $\epsilon/2$. The union $\bigcup_{i=1}^I\bigcup_{l=1}^{L_i}\gamma^{(i)}([\alpha_{i,l}+\epsilon_{i,l},\beta_{i,l}-\epsilon_{i,l}])$ is exactly what we want.
\end{proof}

Roughly speaking, the strategy we adopt in what follows is to show the liminf result with respect to the $\epsilon$-singuarity set, and then pass to the limit $\epsilon \to 0$. 

We start our argument from the next geometric lemma, which tells that the $\mathcal H^1$-measure of a curve can be robustly recovered via box counting. 

\begin{lemma}[A robust box counting lemma]
\label{area_convergence}
    Suppose $\Gamma$ is a {simple} union of $C^1-$curve segments, and for each $n\in\mathbb Z^+$, suppose $S_n\subset\Gamma$ has a measurable $C^1-$preimage and has $\mathcal H^1(S_n)\leq\delta$ for some fixed $\delta>0$, we have estimate
    \begin{equation}
    \liminf_{n\to\infty}\left(\frac{1}{2H_n}\cdot\mathcal H^2(\{x\in\Omega:dist(x,\Gamma\setminus S_n)\leq H_n\})\right)\geq\mathcal H^1(\Gamma)-\delta,
    \end{equation}
    where $\{H_n\}_{n=1}^\infty$ denotes a sequence that satisfies $H_n\to0$.
\end{lemma}

When the exception set $S_n$ is fixed, the theorem is nothing but a standard estimate of a tubular neighbourhood's area when its width shrinks to 0. In our context, however, $S_n$ has to be set variable with respect to resolution $n$ in order to handle the varying sequence $u_n\to u$ involved in the desired $\Gamma$-convergence. Technically, it amounts to characterizing the energy's stability with respect to minor perturbations imposed on the essential singularity $\Gamma^*_{u}$ (within some well-understood potential estimate $\Gamma_{u}$). For the sake of clarity, we postpone its proof to Section \ref{sec_lemmas}, which is technically based on the following two facts. 

\begin{restatable}{lemma}{robustintersection}
\label{robust_intersection}
    For any fixed simple union of $C^1-$curve segments $\Gamma$, there exists a constant $N_\Gamma\in\mathbb Z^+$ such that for any $n\geq N_\Gamma$, if $|\square_{n,k}\cap\Gamma|>0$, there is an $l\in\{-1,0,1\}^2$ such that $|\square_{n,k+l}\cap\Gamma|>2^{-n}/9$.
\end{restatable}
\begin{restatable}{lemma}{boundintersection}
\label{bound_intersection}
    Assume that $\Gamma$ is a simple union of $C^1-$curve segments, there is an $N_\Gamma\in\mathbb Z^+$ and $C_\Gamma>0$ such that $|\square_{n,k}\cap\Gamma|\leq C_\Gamma\cdot2^{-n}$ for any $n\geq N_\Gamma$, $k\in\mathbb K_n$.
\end{restatable}

{The above lemmas bring us the following messages. First, a minor intersection at a box can be augmented by trading off a spatial offset; Second,  $\Gamma$ cannot always be severely intertwined at any single box when resolution $n\to\infty$.}

\subsection{Near-singularity estimates}

We start with a local asymptotic analysis around a fix point $x$ in vicinity of $\Gamma^*_u$, and then conduct estimation over the whole span of this singularity via uniformity yielded from Egorov's theorem.

\begin{proposition}[Coefficient lower-bound near singularity]
\label{large_coeff_convergence}
Assume the model parameters are specified conforming to those stated in {\Cref{assumption_R}} (\Cref{liminf}), then there exists a constant $C>0$ such that for every $x\in\Gamma^*_u$, we have
\begin{equation}
\liminf_{n\to\infty}\frac{c_n[u;\lfloor 2^nx\rfloor+l]}{C\cdot2^n\rho(x)}\geq 1,
\end{equation}
for any $l\in\{-1,0,1\}^2$.
\end{proposition}

In other words, as resolution $n\to\infty$, a single wavelet coefficient near jump discontinuity $x\in\Omega$ has a clean scale separation by factor $2^n$, and, besides, its value should be at least proportionate to the local jump size $\rho(x)$. This characterization is robust with respect to a small spatial perturbation $l$, that is, jumps can be detected consistently by wavelets situated \emph{around} the location.

\begin{proof}
By \eqref{coeff_local}, we have $-I_j2^{-n}c_{n,j}[u,k]=\left<u,\psi^1_{j,n-1,k}\right>$ for $j = 1, 2$ at any $k\in\mathbb K_n$, where $\psi^1_{j,n,k}:=2^{2n}\psi_j(2^n\cdot-k)$ is the $L^1-$normalized wavelet. Denote the rounding error as $a_n:=2^nx-\lfloor 2^nx\rfloor$, we have the following estimation,
\begin{eqnarray*}
&&\left|\left(\sum_{j=1,2}\left|\left<u,\psi^1_{j,n,\lfloor 2^nx\rfloor}\right>\right|^2\right)^{\frac{1}{2}}-\left(\sum_{j=1,2}\left|\left<u^+_x \chi_{y \cdot \nu > 0} + u^-_x \chi_{y \cdot \nu < 0},\psi_j(y - a_n)\right>\right|^2\right)^{\frac{1}{2}}\right|\\
&\leq&\left(\sum_{j=1,2}\left|\left<u,\psi^1_{j,n,\lfloor2^nx\rfloor}\right>-\left<u^+_x \chi_{y \cdot \nu > 0} + u^-_x \chi_{y \cdot \nu < 0} ,\psi_j(y-a_n)\right>\right|^2\right)^{\frac{1}{2}}\\
&=&\left(\sum_{j=1,2}\left|\left<u(2^{-n}y-x)-(u^+_x \chi_{y \cdot \nu > 0} + u^-_x \chi_{y \cdot \nu < 0}),\psi_j(y-a_n)\right>_y\right|^2\right)^{\frac{1}{2}},\\
\end{eqnarray*}
where the last term converges to 0 as $n \to \infty$ by Lemma \ref{fine_property} (in particular, \eqref{jump_local}). 
As a result, it holds that 
$$\left(\sum_{j=1,2}\left|\left<u,\psi^1_{j,n,\lfloor 2^nx\rfloor}\right>\right|^2\right)^{\frac{1}{2}} \to \left(\sum_{j=1,2}\left|\left<u^+_x \chi_{y \cdot \nu > 0} + u^-_x \chi_{y \cdot \nu < 0},\psi_j(y - a_n)\right>\right|^2\right)^{\frac{1}{2}},$$
as $n \to \infty$.

Set $I_0=\max\{|I_1|,|I_2|\}$, we then have
\begin{eqnarray*}
&&\liminf_{n\to\infty}2^{-n}c_n[u,\lfloor2^nx\rfloor]\\
&\geq&\liminf_{n\to\infty}I_0^{-1}\left(\sum_{j=1,2}\left|\left<u,\psi^1_{j,n,k}\right>\right|\right)^{\frac{1}{2}}\\
&=&\liminf_{n\to\infty}I_0^{-1}\left(\sum_{j=1,2}\left|\left<u^+_x \chi_{y \cdot \nu > 0} + u^-_x \chi_{y \cdot \nu < 0},\psi_j(y-a_n)\right>_y\right|^2\right)^{\frac{1}{2}}\\
&\geq&\rho(x)I_0^{-1}\cdot\inf_{a\in[-2,1)^2,\nu\in\mathbb S^1}\left(\sum_{j=1,2}|\left<\chi_{y\cdot\nu>0},\psi_j(y-a)\right>_y|^2\right)^{\frac{1}{2}}\\
&\geq&\rho(x)\cdot C_2I_0^{-1},
\end{eqnarray*}
where $C_2$ is the constant stated in \Cref{assumption_R}, and we conclude the proof by specifying $C:=C_2I_0^{-1}$.
\end{proof}

Now we leverage this local estimate to get a global one, that is, concerning the total energy in the vicinity of the whole (approximate) singularity.

In what follows, we have in mind that $\Gamma_\epsilon$ should denote some appropriately chosen $\epsilon-$approximate to $\Gamma^*_u$ (as the name suggests), though we do not assume this throughout the proof. 

\begin{proposition}[Energy lower-bound, near-singularity]\label{liminf_singular}
For any simple union $\Gamma_\epsilon$ and any sequence $u_n\to u$ in $SBV(\Omega)$, we have asymptotic lower-bound
\begin{equation}
\liminf_{n\to\infty}\left(2^{-2n}\sum_{k\in\widehat\Gamma_{\epsilon,n}}\hat c_n[u_n;k]\right)\geq\alpha\cdot\mathcal H^1(\Gamma_\epsilon\cap\Gamma^*(u)),
\end{equation}
where $\widehat\Gamma_{\epsilon,n}$ denotes the $H_n-$tubular neighbourhood of $\Gamma_\epsilon$, namely,
\begin{equation}\label{tubular}
\widehat\Gamma_{\epsilon,n}:=\{k\in\mathbb K_n:dist(2^{-n}k,\Gamma_\epsilon)\leq H_n\}.
\end{equation}
\end{proposition}
\begin{proof}
For a set $A\subset\Omega$, we denote its cardinal by $|A|_0$. 
The proof is divided into several steps.

\textbf{Step 1.}
Denote by $R:=dist(\Gamma_\epsilon,\partial\Omega)>0$. By the definition of $\mathbb K_n$, for any $k\in\mathbb O_n\setminus\mathbb K_n$ it holds that $dist(2^{-n}k,\partial\Omega)<C_{wav}\cdot2^{-n}$, where $C_{wav}$ is specified according to Assumption \ref{assumption_R}. Since $H_n\to0$, take $N^{(1)}\in\mathbb Z^+$ such that for any $n\geq N^{(1)}$ we have $R>2H_n+C_{wav}\cdot2^{-n}$. 

The first observation is that, the boxes near the singularity $\Gamma_{\epsilon}$ are away from the boundary $\partial \Omega$. Indeed, for any $n$ and $k \in \mathbb K_n$, $dist(2^{-n}k,\Gamma_\epsilon)\leq2H_n$ implies that $dist(2^{-n}k,\partial\Omega)\geq dist(\Gamma_\epsilon,\partial\Omega)-dist(2^{-n}k,\Gamma_\epsilon)>C\cdot2^{-n},$ therefore $$\{k\in\mathbb 
 O_n:dist(2^{-n}k,\Gamma_\epsilon)\leq 2H_n\}\subset\mathbb K_n,$$
 indicating that the boxes near the singular part are in the interior of $\Omega$.

Define $\Gamma^*_\epsilon:=\Gamma_\epsilon\cap\Gamma^*$, clearly the set $\Gamma^*_\epsilon$ has a finite $\mathcal H^1$ measure. For a given threshold $\delta>0$, we take a constant $\eta_\delta>0$ such that
\begin{equation}\label{delta_robust}
\mathcal H^1(\{x\in\Gamma^*_\epsilon:\rho(x)>\eta_\delta\})\geq\mathcal H^1(\Gamma^*_\epsilon)-\delta.
\end{equation}

\textbf{Step 2.}
 Note that Proposition \ref{large_coeff_convergence} tells us the coefficients have a lower bound on the jump set $\Gamma^*$, however, the global estimates require some uniform lower bounds. Thanks to the Egorov's theorem, there exists a set $G_\delta\subset\{x\in\Gamma^*_\epsilon:\rho(x)>\eta_\delta\}$ together with a constant $N^{(2)}_\delta\in\mathbb Z^+$ such that
\begin{equation}\label{delta_convergence}
\mathcal H^1(G_\delta)\geq\mathcal H^1(\{x\in\Gamma^*_\epsilon:\rho(x)>\eta_\delta\})-\delta
\end{equation}
and
\begin{equation}\label{large_coeff}
2^{-n}c_n[u;\lfloor2^nx\rfloor+l]\geq (C_1/2)\cdot\eta_\delta
\end{equation}
for any $x\in G_\delta$, $n\geq N^{(2)}_\delta$ and $l\in\{-1,0,1\}^2$, where $C_1>0$ is a constant independent to $n$. Recall that $H_n=2^{-n/2}$ and $M_n=\frac{\alpha}{2H_n}$. It is straightforward to see that
$
\lim_{n\to\infty}\frac{(C_1/2)\cdot\eta_\delta\cdot2^n}{M_n}=+\infty,
$
and therefore we can find $N^{(3)}_\delta\in\mathbb Z^+$ such that \begin{equation}\label{estimate_m}(C_1/2)\cdot\eta_\delta\cdot2^n/M_n\geq2
\end{equation}
for any $n\geq N^{(3)}_\delta$.

\textbf{Step 3.} We now modify the set $G_{\delta}$ to a more regular set $\tilde{G}_{\delta}$, and measure the discrepancy brought about.
Let $\tilde G_\delta\subset\Gamma_\epsilon$ be a union of closed curve segments such that $\mathcal H^1(\tilde G_\delta\setminus G_\delta)<\delta$ and $\mathcal H^1(G_\delta\setminus\tilde G_\delta)<\delta$. In particular,
\begin{equation}\label{delta_segments}
\mathcal H^1(\tilde G_\delta)>\mathcal H^1(G_\delta)-\delta.
\end{equation}
We define the following index sets
\begin{equation}
\mathbb G_{\delta,n}:=\{\lfloor2^nx\rfloor:x\in G_\delta\},\ \ \ \ \tilde{\mathbb G}_{\delta,n}:=\{\lfloor2^nx\rfloor:x\in\tilde G_{\delta}\},
\end{equation}
and the auxiliary index set
\begin{equation}
\tilde {\mathbb G}^\circ_{\delta,n}:=\{k\in\tilde{\mathbb G}_{\delta,n}:\mathcal H^1(\square_{n,k}\cap\tilde G_\delta)>2^{-n}/9\},
\end{equation}
respectively, and for a (discrete index) set $A\subset\Omega$, denote $A^+:=A+\{-1,0,1\}^2$.

The goal of this (and the next) step is to show that almost all of the boxes in $\tilde {\mathbb G}_{\delta,n}$ have coefficients larger than the threshold $M_n$. 
Since $\mathcal H^1(\tilde G_\delta\setminus G_\delta)<\delta$, we have
\begin{equation}
|\tilde{\mathbb G}^\circ_{\delta,n}\setminus\mathbb G_{\delta,n}|_0\leq 9\delta\cdot 2^n.
\end{equation}

As a consequence, we have
\begin{equation}
|(\tilde{\mathbb G}^\circ_{\delta,n})^+\setminus(\mathbb G_{\delta,n})^+|_0\leq81\delta\cdot 2^n.
\end{equation}
By Lemma \ref{robust_intersection}, $\tilde{\mathbb G}_{\delta,n}\subset(\tilde{\mathbb G}^\circ_{\delta,n})^+$, therefore
\begin{equation}
|\tilde{\mathbb G}_{\delta,n}\setminus(\mathbb G_{\delta,n})^+|_0\leq81\delta\cdot2^n.
\end{equation}

\textbf{Step 4.}
Since $u_n\to u$ in $SBV(\Omega)$, there is a constant $N^{(4)}_\delta\in\mathbb Z^+$ such that $\|u_n-u\|_{SBV}<\eta_\delta\cdot\delta$ for any $n\geq N^{(4)}_\delta$. Apply Lemma~\ref{coeff_sum} to $u_n$ and $u$, we then have
\begin{equation}
2^{-2n}\sum_{k\in\mathbb K_n}|c_n[u_n;k]-c_n[u;k]|<C_2\cdot\eta_\delta\cdot\delta,
\end{equation}
where $C_2>0$ is a constant independent to $\delta$ and $n$. By definition of $\left(\mathbb G_{\delta,n}\right)^+$, \eqref{large_coeff} is amount to saying that $c_n[u;k]\geq (C_1/2)\cdot\eta_\delta\cdot2^n$ for any $k\in\left(\mathbb G_{\delta,n}\right)^+$. Therefore
\begin{eqnarray*}
&&|\{k\in\left(\mathbb G_{\delta,n}\right)^+:c_n[u_n;k]<M_n\}|_0\\
&\leq&\frac{2^{2n}C_2\cdot\eta_\delta\cdot\delta}{(C_1/2)\cdot\eta_\delta\cdot2^n-M_n}\\
&\leq&\frac{2^{2n}C_2\cdot\eta_\delta\cdot\delta}{(C_1/4)\cdot\eta_\delta\cdot2^n}\\
&=&4C_2C_1^{-1}\cdot\delta\cdot2^n
\end{eqnarray*}
for sufficiently large $n$.
Here the second line comes from the fact that $|c_n[u;k] - c_n[u_n;k]|\ge (C_1/2) \cdot \eta_{\delta} \cdot 2^n - M_n$, while the third line comes from \eqref{estimate_m}.

Denote $\mathbb G'_{\delta,n}:=\{k\in\tilde{\mathbb G}_{\delta,n}:c_n[u_n;k]\geq M_n\}$ and $\mathbb E_{\delta,n}:=\tilde{\mathbb G}_{\delta,n}\setminus\mathbb G'_{\delta,n}$. Note there $\mathbb E_{\delta, n}$ collects the box indices whose coefficients are negligible, compared to the threshold $M_n$. A combination of the previous two steps indicates that $\mathbb E_{\delta, n}$ is a small set. Concretely, we have $|\mathbb E_{\delta,n}|_0\leq C_3\delta\cdot2^n,$ where $C_3:=81+4C_2C_1^{-1}>0$.

\textbf{Step 5.} We now go back to estimate the corresponding region, with the help of the above result about the index sets. 
Denote $G'_{\delta,n}:=\tilde G_\delta\setminus\bigcup_{k\in\mathbb E_{\delta,n}}\square_{n,k}$, and we have
\begin{eqnarray*}
\mathcal H^1(G'_{\delta,n})&\geq&\mathcal H^1(\tilde G_\delta)-\sum_{k\in\mathbb E_{\delta,n}}\mathcal H^1(\tilde G_\delta\cap\square_{n,k})\\
&\geq&\mathcal H^1(\tilde G_\delta)-|\mathbb E_{\delta,n}|_0\cdot(C_4\cdot2^{-n})\\
&\geq&\mathcal H^1(\tilde G_\delta)-C_5\delta,
\end{eqnarray*}
where we applied Lemma \ref{bound_intersection}.

Denote $\widehat G'_{\delta,n}:=\{x\in\Omega:dist(x,G'_{\delta,n})\leq H_n-4\sqrt2\cdot2^{-n}\}$. Apply Lemma \ref{area_convergence} (and notice that $\lim_{n\to\infty}\frac{H_n-4\sqrt2\cdot2^{-n}}{H_n}=1$ following the definition of $H_n$), we have 
\begin{equation}
\liminf_{n\to\infty}\frac{\mathcal H^2(\widehat G'_{\delta,n})}{2H_n}=\liminf_{n\to\infty}\frac{\mathcal H^2(\widehat{G}'_{\delta,n})}{2(H_n-4\sqrt{2}\cdot2^{-n})}\geq\mathcal H^1(\tilde G_\delta)-C_5\delta,
\end{equation}
thus there is an $N^{(5)}_\delta$ such that $\mathcal H^2(\widehat G'_{\delta,n})/(2H_n)\geq \mathcal H^1(\tilde G_\delta)-(1+C_5)\delta$ for any $n\geq N^{(5)}_\delta$.

On the other hand, if $x\in\widehat G'_{\delta,n}$, we have
\begin{equation}
2^{-n}\cdot dist(\lfloor2^nx\rfloor,\mathbb G'_{\delta,n})\leq H_n-2\sqrt2\cdot2^{-n},
\end{equation}
therefore $\lfloor2^nx\rfloor\in\widehat{\mathbb G}'_{\delta,n}:=\{k\in\mathbb K_n:2^{-n}\cdot dist(k,\mathbb G'_{\delta,n})\leq H_n-2\sqrt2\cdot2^{-n}\}$. (Note that $k\in\widehat{\mathbb G}'_{\delta,n}$ implies $\hat c_n[u_n,k]\geq M_n$.) Consequently, $\widehat G'_{\delta,n}\subset\bigcup_{k\in\widehat{\mathbb G}'_{\delta,n}}\square_{n,k}$.

\textbf{Step 6.} To complete the proof, it remains to show that $\widehat{\mathbb G}'_{\delta,n}\subset\widehat\Gamma_{\epsilon,n}$. Indeed, by definition of $\tilde{\mathbb G}_{\delta,n}$ and the fact $\tilde G_\delta\subset\Gamma_\epsilon$, for any $l\in\tilde{\mathbb G}_{\delta,n}$, $dist(2^{-n}l,\Gamma_\epsilon)\leq dist(2^{-n}l,\tilde G_{\delta})\leq \sqrt2\cdot2^{-n}$. Therefore, note that $\mathbb G'_{\delta,n}\subset\tilde {\mathbb G}_{\delta,n}$, if $k\in\widehat{\mathbb G}'_{\delta,n}$, we have $dist(2^{-n}k,\Gamma_\epsilon)\leq (H_n-2\sqrt2\cdot2^{-n})+\sqrt2\cdot2^{-n}<H_n$ and thus the desired inclusion. Therefore
\begin{eqnarray*}
&&2^{-2n}\sum_{k\in\widehat\Gamma_{\epsilon,n}}\hat c_n[u_n;k]\\
&\geq&2^{-2n}\sum_{k\in\widehat{\mathbb G}'_{\delta,n}}\hat c_n[u_n;k]\\
&\geq&2^{-2n}\left|\widehat{\mathbb G}'_{\delta,n}\right|_0\cdot M_n\\
&=&\mathcal H^2\left(\bigcup_{k\in\widehat{\mathbb G}'_{\delta,n}}\square_{n,k}\right)\cdot\frac{\alpha}{2H_n}\\
&\geq&\mathcal H^2(\widehat G'_{\delta,n})\cdot\frac{\alpha}{2H_n}\\
&\geq&\alpha\left(\mathcal H^1(\tilde G_\delta)-(1+C_5)\delta\right)\\
&\geq&\alpha\ (\mathcal H^1(\Gamma^*_\epsilon)-(4+C_5)\delta).
\end{eqnarray*}
where we applied \eqref{delta_robust}, \eqref{delta_convergence} and \eqref{delta_segments} to conclude the last inequality. Since all the above argument is established for all $n\geq N(\delta)$ (where $N(\delta):=\max\{N^{(1)},N^{(2)}_\delta,N^{(3)}_\delta,N^{(4)}_\delta\}$, in particular), and $C_5$ is a plain constant, we conclude this proposition.
\end{proof}

\subsection{Off-singularity estimates}

Now we bound the energy induced by $\nabla^au$, the absolutely continuous component of the variation, which is negligible by order near the narrow singularity, but makes dominant contributions elsewhere.

\begin{proposition}[Energy lower-bound, off-singularity]\label{liminf_nonsingular}

Suppose the simple union $\Gamma_\epsilon$ satisfies residue condition $\int_{\Omega\setminus\Gamma_\epsilon}|\nabla^su|<\epsilon$, $\widehat\Gamma_{\epsilon,n}$ denotes its tubular neighbourhood by \eqref{tubular}, and $\mathbb K'_n:=\mathbb K_n\setminus\widehat{\Gamma}_{\epsilon,n}$, the lattice complement, then, for any sequence $u_n\to u$ in $SBV(\Omega)$, we have
\begin{equation}
\liminf_{n\to\infty}\left(2^{-2n}\sum_{k\in\mathbb K'_n}\hat c_n[u_n;k]\right)\geq\int_{\Omega\setminus\Gamma^*(u)}|\nabla u|-C\cdot\epsilon.
\end{equation}
\end{proposition}
Recall these $\hat c$-coefficients are processed with \emph{thresholding} and \emph{neighbourhood maximum} operations apart from a standard $\ell^2-$summation across filter indices. Note $c_{n,j}[u;k]=\left<\partial_ju,I_j^{-1}\varphi^1_{j,n,k}\right>$ by \eqref{coeff_local}, we adopt here the notation $\bar\varphi_{j,n,k}:=I_j^{-1}\varphi^1_{j,n,k}$ for convenience.
Our proof sketch is summarized as follows: 
\begin{itemize}
    \item[-] Consider the target function $u\in SBV(\Omega)$ apart from the varying sequence $\{u_n\}$ (that converges to it). We control the energy difference between its regular coefficients $c_{n,j}[u;k]$ and the \emph{$a$-coefficients} $c^a_{n,j}[u;k]:=\left<\partial^a_ju,\bar\varphi_{j,n,k}\right>$, the part of contribution induced only from its absolutely continuous gradient $\nabla^au$. In Lemma \ref{loss_singular}, we will establish this bound in terms of the approximation residue $\epsilon$ (of the singular gradient $\nabla^su$'s variation).
    \item[-] Then we consider the loss generated from truncating $|c^a_{n,j}|$ at asymptotic threshold $M_n/\sqrt{2}$. By Lemma \ref{loss_truncation}, we bound this quantity by enforcing the integrability of $\nabla^a u$ (respectively, the absolute continuity of the measure $\nabla^audx$).
    \item[-] Finally, we get from $\min \{M_n/\sqrt{2}, |c_{n,j}^a| \}$ back to $\hat{c}_{n,j} := \min \{M_n/\sqrt{2}, |c_{n,j}| \}$ with another $\epsilon$ tolerance, and finish our proof by enforcing coefficients stability (Lemma \ref{coeff_sum}, to the sequence $u_n\to u$) together with the discretization loss bound (Lemma \ref{loss_discretization}).
\end{itemize}

We present each of these steps individually as a lemma, and finally explain how they together imply the desired proposition. 
\begin{lemma}\label{loss_singular}
On the partial lattice $\mathbb K'_n$, the difference between sums calculated from $c_{n,j}[u;k]$ and its non-singular component $c^a_{n,j}[u;k]$ is bounded as
$$\left|2^{-2n}\sum_{k\in\mathbb K'_n}\left(\sum_{j=1,2}|c^a_{n,j}[u;k]|^2\right)^{\frac{1}{2}}-2^{-2n}\sum_{k\in\mathbb K'_n}\left(\sum_{j=1,2}|c_{n,j}[u;k]|^2\right)^{\frac{1}{2}}\right| \leq C\cdot\epsilon.
$$
\end{lemma}
\begin{proof} 

First notice that 
\begin{eqnarray*}
&&\left|2^{-2n}\sum_{k\in\mathbb K_n\setminus\widehat\Gamma_{\epsilon,n}}\left(\sum_{j=1,2}|c^a_{n,j}[u;k]|^2\right)^{\frac{1}{2}}-2^{-2n}\sum_{k\in\mathbb K_n\setminus\widehat\Gamma_{\epsilon,n}}\left(\sum_{j=1,2}|c_{n,j}[u;k]|^2\right)^{\frac{1}{2}}\right|\\
&\leq&2^{-2n}\sum_{k\in\mathbb K_n\setminus\widehat\Gamma_{\epsilon,n}}\left(\sum_{j=1,2}|c^a_{n,j}[u;k]-c_{n,j}[u;k]|^2\right)^{\frac{1}{2}}\\
&\leq&2^{-2n}\sum_{k\in\mathbb K_n\setminus\widehat\Gamma_{\epsilon,n}}\sum_{j=1,2}\left<|\partial^s_ju|,\bar\varphi_{j,n,k}\right>.
\end{eqnarray*}

Here $\partial_j^s u$ denotes $\nabla^s u \cdot e_j$. As a result, we have

\begin{eqnarray*}
&& 2^{-2n}\sum_{k\in\mathbb K_n\setminus\widehat\Gamma_{\epsilon,n}}\sum_{j=1,2}\left<|\partial^s_ju|,\bar\varphi_{j,n,k}\right>\\
&\leq&\left<|\nabla^su|,\sum_{j=1,2}\sum_{k\in\mathbb K_n\setminus\widehat\Gamma_{\epsilon,n}}2^{-2n}\bar\varphi_{j,n,k}\right> \\
&\leq&\left<|\nabla^su|,C\cdot\chi_{\{x\in\Omega:dist(x,\Gamma_\epsilon)\geq 2^{-n}\}}\right> \\
&\leq&C\cdot\epsilon.
\end{eqnarray*}
Here the third line comes from the partition of unity property of $\bar\varphi_{j,n,k}$, its (pointwise) non-negativity and the fact that its support radius is smaller than $c_\varphi2^{-n}$ for some constant $c_\varphi$, and the last line comes from \eqref{residue}. 
\end{proof}

Next, we consider the effects of the truncation.
\begin{lemma}\label{loss_truncation} Concerning $a$-coefficients $c^a_{n,j}$ and its truncation, it holds that 
$$\Xi_n:=\left|2^{-2n}\sum_{k\in\mathbb K_n}\left(\sum_{j=1,2}|c^a_{n,j}[u;k]|^2\right)^{\frac{1}{2}}-2^{-2n}\sum_{k\in\mathbb K_n}\left(\sum_{j=1,2}\min\left\{\frac{M_n}{\sqrt2},|c^a_{n,j}[u;k]|\right\}^2\right)^{\frac{1}{2}}\right| \to 0$$ as $n \to \infty$.
\end{lemma}

\begin{proof}

For any $M, c_1, c_2 >0$, we have the following estimation
\begin{equation}
\min\left\{M,\sqrt{c_1^2+c_2^2}\right\}\geq\sqrt{\min\{M/\sqrt2,c_1\}^2+\min\{M/\sqrt2,c_2\}^2},
\end{equation}
which can be verified straightforwardly. 

By the triangle inequality, it holds that
\begin{eqnarray*}
    &&\left|2^{-2n}\sum_{k\in\mathbb K_n}\left(\sum_{j=1,2}|c^a_{n,j}[u;k]|^2\right)^{\frac{1}{2}}-2^{-2n}\sum_{k\in\mathbb K_n}\left(\sum_{j=1,2}\min\left\{\frac{M_n}{\sqrt2},|c^a_{n,j}[u;k]|\right\}^2\right)^{\frac{1}{2}}\right|\\
&\leq&2^{-2n}\left|\sum_{k\in\mathbb K_n}\left(\sum_{j=1,2}\left||c^a_{n,j}[u;k]|-\min\left\{\frac{M_n}{\sqrt2},|c^a_{n,j}[u;k]|\right\}\right|^2\right)^\frac{1}{2}\right|.
\end{eqnarray*}
Since a vector's $\ell^2$-norm is no more than its $\ell^1$-norm, it suffices to estimate the following quantity,
\[2^{-2n}\sum_{j=1,2}\sum_{k\in\mathbb K_n}\left(|c^a_{n,j}[u;k]|-\min\left\{\frac{M_n}{\sqrt2},|c^a_{n,j}[u;k]|\right\}\right),\]
which, by definition, is equal to
\[2^{-2n}\sum_{j=1,2}\sum_{k\in\mathbb K_n}\left(\left|\left<\partial^a_ju,\bar\varphi_{j,n,k}\right>\right|-\min\left\{\frac{M_n}{\sqrt2},\left|\left<\partial^a_ju,\bar\varphi_{j,n,k}\right>\right|\right\}\right).\]

Since $\bar\varphi_{j,n,k}$ is non-negative, the mapping $f \mapsto \langle f, \bar\varphi_{j,n,k} \rangle$ is monotone\footnote{A linear functional $L$ is monotone, if $f \ge g$ implies $L(f) \ge L(g)$.}. Combining with the fact that $x - \min\{M_n/\sqrt{2}, x\}$ is increasing, we obtain that 
\begin{eqnarray*}
    && 2^{-2n}\sum_{j=1,2}\sum_{k\in\mathbb K_n}\left(\left|\left<\partial^a_ju,\bar\varphi_{j,n,k}\right>\right|-\min\left\{\frac{M_n}{\sqrt2},\left|\left<\partial^a_ju,\bar\varphi_{j,n,k}\right>\right|\right\}\right) \\ 
    &\leq&2^{-2n}\sum_{j=1,2}\sum_{k\in\mathbb K_n}\left(\left<|\partial^a_ju|,\bar\varphi_{j,n,k}\right>-\min\left\{\frac{M_n}{\sqrt2},\left<|\partial^a_ju|,\bar\varphi_{j,n,k}\right>\right\}\right),
\end{eqnarray*}
since $\partial^a_ju \le |\partial^a_ju|$.

The remaining proof is based on converting the coefficients to the functions. To this end, set $$v_j(x):=|\partial^a_ju(x)|,$$
\begin{equation}
v_{j,n}(x):=\sum_{k\in\mathbb K_n}\left<|\partial^a_ju|,\bar\varphi_{j,n,k}\right>\chi_{\square_{n,k}}(x),
\end{equation}
and their truncated counterparts,
\begin{equation}
\tilde v^{(n)}_j(x):=\min\{M_n/\sqrt2,v_j(x)\},
\end{equation}
\begin{equation}
\tilde v^{(n)}_{j,n}(x):=\min\{M_n/\sqrt2,v_{j,n}(x)\},
\end{equation}
respectively.
Thus we can rewrite our upper-bound in integral form as
$$2^{-2n}\sum_{j=1,2}\sum_{k\in\mathbb K_n}\left(\left<|\partial^a_ju|,\bar\varphi_{j,n,k}\right>-\min\left\{\frac{M_n}{\sqrt2},\left<|\partial^a_ju|,\bar\varphi_{j,n,k}\right>\right\}\right)=\sum_{j = 1,2} \|v_{j,n} - \tilde v_{j,n}^{(n)}\|_{L^1(\Omega)}\ \ =:\Xi'_n.$$

We then consider the telescope expansion 
$$ \sum_{j=1,2}\|v_{j,n}-\tilde v^{(n)}_{j,n}\|_{L^1(\Omega)}
 \leq \sum_{j=1,2}\left(\|v_{j,n}-v_j\|_{L^1(\Omega)}+\|v_j-\tilde v^{(n)}_j\|_{L^1(\Omega)}+\|\tilde v^{(n)}_j-\tilde v^{(n)}_{j,n}\|_{L^1(\Omega)}\right).
 $$

By \Cref{approximation_theorem}, $\|v_{j,n}-v_j\|_{L^1(\Omega)}\to0$. Moreover, since the truncation operator is contractive, it follows that $\|\tilde v^{(n)}_j-\tilde v^{(n)}_{j,n}\|_{L^1(\Omega)}\leq\|v-v_{j,n}\|_{L^1(\Omega)}$ and therefore $\|\tilde v^{(n)}_j-\tilde v^{(n)}_{j,n}\|_{L^1(\Omega)}\to0$. Since $v_j\in L^1(\Omega)$ and $v_j(x)\in [0,+\infty)$ for a.e. $x\in\Omega$, the second term $\|v_j-\tilde v^{(n)}_j\|_{L^1(\Omega)}\to0$ by definition of Lebesgue integration. In summary, we conclude that our bound $\Xi'_n\to0$, therefore $\Xi_n\to0$.

\end{proof}

Next, we show that the coefficients $c_{n,j}^a$ are indeed an effective discretization of the absolutely continuous gradient $\nabla^a u$.
\begin{lemma}\label{loss_discretization} The term 
$$\Theta_n := \left|\sum_{k\in\mathbb K'_n}\left(\sum_{j=1,2}|c^a_{n,j}[u;k]|^2\right)^{\frac{1}{2}} - \int_\Omega|\nabla^au| \right|$$  converges to $0$ as $n \to \infty.$ 
\end{lemma}

\begin{proof}
Note that $\partial^a_ju\in L^1(\Omega)$, it again follows from \Cref{approximation_theorem} that (for each $j=1,2$)
\begin{equation*}
\lim_{n\to\infty}\left\|\sum_{k\in\mathbb O_n}\left<\partial^a_ju,\bar\varphi_{j,n,k}\right>\chi_{\square_{n,k}}-\partial^a_ju\right\|_{L^1(\Omega)}\to0.
\end{equation*}
Denote $\mathbb K'_n:=\mathbb K_n\setminus\widehat\Gamma_{\epsilon,n}$ and $K'_n:=\bigcup_{k\in\mathbb K'_n}\square_{n,k}$, since 
\begin{eqnarray*}
    &&\left|\sum_{k\in\mathbb K'_n}\left(\sum_{j=1,2}|c^a_{n,j}[u;k]|^2\right)^{\frac{1}{2}}-\int_\Omega|\nabla^au|\right|\\
&\leq&\left|\sum_{k\in\mathbb K'_n}\left(\sum_{j=1,2}|c^a_{n,j}[u;k]|^2\right)^{\frac{1}{2}}-\int_{K'_n}|\nabla^au|\right|+\int_{\Omega\setminus K'_n}|\nabla^au|
\end{eqnarray*}
The second term at the right hand side vanishes since the Lebesgue measure $|\Omega\setminus K_n'|\to0$ as $n \to \infty$, and $|\nabla^au|$ is absolutely continuous by definition. Therefore, it suffices to estimate the first term, with which we have:
\begin{eqnarray*}
&&\left|\sum_{k\in\mathbb K'_n}\left(\sum_{j=1,2}|c^a_{n,j}[u;k]|^2\right)^{\frac{1}{2}}-\int_{K'_n}|\nabla^au|\right|\\
&=&\left|\int_{K'_n}\left(\sum_{j=1,2}\left|\sum_{k\in\mathbb O_n}\left<\partial^a_ju,\bar\varphi_{j,n,k}\right>\chi_{\square_{n,k}}\right|^2\right)^{\frac{1}{2}}-\int_{K'_n}\left(\sum_{j=1,2}|\partial^a_ju|^2\right)^{\frac{1}{2}}\right|\\
&\leq&\int_{K'_n}\sum_{j=1,2}\left|\sum_{k\in\mathbb K'_n}\left<\partial^a_ju,\bar\varphi_{j,n,k}\right>\chi_{\square_{n,k}}-\partial^a_ju\right|\\
&\leq&\sum_{j=1,2}\int_\Omega\left|\sum_{k\in\mathbb K'_n}\left<\partial^a_ju,\bar\varphi_{j,n,k}\right>\chi_{\square_{n,k}}-\partial^a_ju\right|
\end{eqnarray*}

Since $\partial^a_ju\cdot dx$ is absolutely continuous for $j=1,2$, the measure $|\nabla^au|dx$ is also absolutely continuous. 
Consequently we have $\Theta_n \to 0$.
\end{proof}

We are now ready to put all the pieces together to obtain our global off-singularity control.

\begin{proof}[Proof of \Cref{liminf_nonsingular}]

Assume $\delta>0$, let $N_\delta\in\mathbb Z^+$ be a number such that for any $n>N_\delta$, we have $\Xi_n,\Theta_n<\delta$ (existence guaranteed by Lemma \ref{loss_truncation} and Lemma \ref{loss_discretization} together), and $\|u_n-u\|_{SBV(\Omega)}<\delta$ (fulfilled as $u_n\to u$). In this case, we have
\begin{eqnarray*}
&&2^{-2n}\sum_{k\in\mathbb K'_n}\hat c_n[u_n;k]\\
&=&2^{-2n}\sum_{k\in\mathbb K'_n}\min\{M_n,c_n[u_n;k]\}\\
&\geq&2^{-2n}\sum_{k\in\mathbb K'_n}\min\{M_n,c_n[u;k]\}-2^{-2n}\sum_{k\in\mathbb K_n}|c_n[u_n;k]-c_n[u;k]| \\
&\geq&2^{-2n}\sum_{k\in\mathbb K'_n}\min\{M_n,c_n[u;k]\} - C_1 \delta,
\end{eqnarray*}
where Lemma \ref{coeff_sum} is applied in order to obtain the last inequality.

By triangular inequality and Lemma \ref{loss_singular} it holds that 
\begin{eqnarray*}
    &&2^{-2n}\sum_{k\in\mathbb K'_n}\min\{M_n,c_n[u;k]\} \\
    &\geq&2^{-2n}\sum_{k\in\mathbb K'_n}\min\{M_n,c^a_n[u;k]\} -2^{-2n}\sum_{k\in\mathbb K'_n}|c^a_n[u;k]-c_n[u;k]| \\
    &\geq&2^{-2n}\sum_{k\in\mathbb K'_n}\left(\sum_{j=1,2}\min\left\{\frac{M_n}{\sqrt2},|c^a_{n,j}[u;k]|\right\}^2\right)^{\frac{1}{2}}-C\cdot\epsilon 
\end{eqnarray*}

For any $k\in\mathbb K_n$, it follows that 
\begin{equation}
\sum_{k\in\mathbb K'_n}\left(\left(\sum_{j=1,2}|c^a_{n,j}[u;k]|^2\right)^{1/2}-\left(\sum_{j=1,2}|\min\{M_n/\sqrt{2},c^a_{n,j}[u;k]\}|^2\right)^{1/2}\right)\ \ \leq\Xi_n<\delta
\end{equation}
by our assumption on $\delta$. Moreover,
\begin{eqnarray*}
      &&2^{-2n}\sum_{k\in\mathbb K'_n}\left(\sum_{j=1,2}\min\left\{\frac{M_n}{\sqrt2},|c^a_{n,j}[u;k]|\right\}^2\right)^{\frac{1}{2}}-C\cdot\epsilon \\ 
      &\geq&2^{-2n}\sum_{k\in\mathbb K'_n}\left(\sum_{j=1,2}|c^a_{n,j}[u;k]|^2\right)^{\frac{1}{2}}-(1+C_1)\delta-C\cdot\epsilon.
\end{eqnarray*}
Finally, since $\Theta_n<\delta$, we have
\begin{eqnarray*}
    && 2^{-2n}\sum_{k\in\mathbb K'_n}\left(\sum_{j=1,2}|c^a_{n,j}[u;k]|^2\right)^{\frac{1}{2}}-(1+C_1)\delta-C\cdot\epsilon \\ 
    &\geq&\int_{\Omega\setminus\Gamma^*}|\nabla u|-(2+C_1)\delta-C\cdot\epsilon,
\end{eqnarray*}
and eventually we obtain our desired limit-infimum. 
\end{proof}

\subsection{Asymptotic lower-bound: the universal $\Gamma$-property}
Totally, we have the following asymptotic bound for the regularity term.
\begin{proposition}
For any sequence $u_n\to u$ in $SBV(\Omega)$, we have
\begin{equation}
\liminf_{n\to\infty}R_n(u_n)\geq R(u),
\end{equation}
where $R_n$ and $R$ denote regularity terms defined by \eqref{definition_R} and \eqref{definition_R_n} respectively.
\end{proposition}
\begin{proof}
Assume $\epsilon>0$. We separately consider two cases concerning the (essential) jump singularity $\Gamma^*(u)$, in which it has finite and infinite $\mathcal H^1$-measures respectively.

\textbf{Case 1.} If $\mathcal H^1(\Gamma^*(u))<\infty$, we define $\Gamma_\epsilon$ as an $\epsilon$-singularity set such that \eqref{finite_curve}-\eqref{residue} are satisfied. In particular, $\mathcal H^1(\Gamma^*(u)\cap\Gamma_\epsilon)>(1-\epsilon)\mathcal H^1(\Gamma^*(u))$ as is specified in \eqref{coverage}. Consequently, we define $\widehat\Gamma_{\epsilon,n}$ according to \eqref{tubular}. Combining Proposition \ref{liminf_singular} and Proposition \ref{liminf_nonsingular}, we get
\begin{eqnarray}
&&\liminf_{n\to\infty}R_n(u_n)\\
&=&\liminf_{n\to\infty}\left(2^{-2n}\sum_{k\in\mathbb K_n}\hat c_n[u_n;k]\right)\\
&\geq&\liminf_{n\to\infty}\left(2^{-2n}\sum_{k\in\mathbb K_n\setminus\widehat\Gamma_{\epsilon,n}}\hat c_n[u_n;k]\right)+\liminf_{n\to\infty}\left(2^{-2n}\sum_{k\in\widehat\Gamma_{\epsilon,n}}\hat c_n[u_n;k]\right)\\
&\geq&\left(\int_{\Omega\setminus\Gamma^*(u)}|\nabla u|-C\cdot\epsilon\right)+\alpha\cdot\mathcal H^1(\Gamma^*(u)\cap\Gamma_\epsilon)\\
&\geq&R(u)-(C+\alpha\cdot\mathcal H^1(\Gamma^*(u)))\cdot\epsilon.
\end{eqnarray}
In the lase inequality, the constant multiple of $\epsilon>0$ can be arbitrarily small, thus we conclude our limit-infimum.

\textbf{Case 2.} If $\mathcal H^1(\Gamma^*(u))=\infty$, we define $\Gamma_\epsilon$ and $\widehat\Gamma_{\epsilon,n}$ accordingly. In particular, we specify $\mathcal H^1(\Gamma^*(u)\cap\Gamma_\epsilon)>\epsilon^{-1}$ instead, by \eqref{coverage}. Then we apply Proposition \ref{liminf_singular} only, and achieve
\begin{eqnarray}
&&\liminf_{n\to\infty}R_n(u_n)\\
&\geq&\liminf_{n\to\infty}\left(2^{-2n}\sum_{k\in\widehat\Gamma_{\epsilon,n}}\hat c_n[u_n;k]\right)\\
&\geq&\alpha\cdot\epsilon^{-1}.
\end{eqnarray}
Since $\epsilon^{-1}>0$ can be arbitrarily large, we obtain our desired divergence.
\end{proof}

Then Proposition~\ref{liminf} follows from the above proposition and Theorem~\ref{limfid}.

\subsection{Tightness argument: the existence $\Gamma$-property
}\label{section_tightness}

In the rest of this section, we complete the other half of our desired $\Gamma$-convergence. With the help of an approximation proposition introduced in \cite{de2017approximation}, we strive to prove Proposition \ref{limsup}, which says there does exist a sequence that attains (thus validates) the proposed $\Gamma$-limit.

Recall that if $u\in SBV(\Omega)$, its essential singularity set, denoted $\Gamma^*$ (or $\Gamma^*_u$, if necessary), must be $C^1$-rectifiable. If its Hausdorff measure $\mathcal H^1(\Gamma^*)=\infty$, then we immediately have $R_n(u)=\infty$ and our purpose is fulfilled simply by defining $u^*_n:=u$ for all $n$. We thus concentrate only on the more interesting case $\mathcal H^1(\Gamma^*)<\infty$.

By Lemma \ref{curve_truncation}, we first take a \emph{simple union} $\Gamma_m=\Gamma_m(u)$ (\emph{i.e.} finite and without intersections) at each precision level $m$ in the domain $\Omega$, such that approximation properties: $\mathcal H^1(\Gamma_m)<\mathcal H^1(\Gamma^*)+2^{-m}$, $\mathcal H^1(\Gamma^*\setminus\Gamma_m)<2^{-m}$ and $\int_{\Omega\setminus\Gamma_m}|\nabla^su|<2^{-m}$, are satisfied. We first mollify off all singularity of $u$ apart from $\Gamma_m$, getting an object called $u'_m$, and then perform a further surgery on each $u'_m$ near $\Gamma_m$, resulting in $u''_m$, whose gradient possesses necessary boundedness away from singularity. Finally, we rearrange the index and get our desired sequence $u^*_n.$

Since $\Gamma_m$ is a finite union of $C^1$-curves, there exists $\delta_m < 2^{-m}$ such that the distance function 
\[
d_m(x) = dist(x, \Gamma_m)
\]
is $C^1$-regular where $d_m(x) < 3\delta_m$. Consider a spline $g_m \in C^{\infty}(\mathbb R^+)$ such that $g_m(d) = d$ when $d \in [0,\delta_m]$, but $g_m(d) = 2\delta_{m}$ when $d \in [2\delta_m,+\infty)$, and $g_m'(d) < 2$ for all $d$. Define $r_m(x) = g_m(d_m(x))$, and 
\begin{equation}\label{eqn_conv}
u'_m(x)=u*'_m\eta\ (x):=\int_{\mathbb R^2}u(x-r_m(x)y)\eta(y)dy.
\end{equation}

This operation guarantees a number of desirable asymptotic properties that are comparable to those of ordinary convolution (with a fixed kernel). Relevant details are postponed to Appendix \ref{proof_sequence_properties}, and the interested reader may refer to \cite{de2017approximation} for a more extensive account of this technique.

\begin{proposition}\label{sequence_properties}
\cite{de2017approximation} Concerning the mollifying operation $-*'_m\eta$ defined through Equation \eqref{eqn_conv}, the following statements hold,
\begin{enumerate}
\item For any $f\in L^1(\Omega)$, we have $\lim_{m\to\infty}\|f*'_m\eta-f\|_{L^1(\Omega)}=0$.
\item Assume $u\in SBV(\Omega)$, then for each $m\in\mathbb Z^+$, we have $u'_m\in SBV(\Omega)$. In particular, its non-trivial jump set $\Gamma^*(u'_m)\subset\Gamma_m$, equivalently, $\nabla u'_m=\nabla^a u'_m$ on $\Omega\setminus\Gamma_m$. 
\item The jump singularity of $u'_m$ is equal to that of $u$ as a distribution, that is $\nabla^s u'_m=\nabla^s u$, on $\Omega$. In particular, $supp(\nabla^s u'_m)=supp(\nabla^su)=\Gamma_m$.
\item Let $u$, $u'_m$ be defined as above. There is convergence
\begin{equation*}
\lim_{m\to\infty}\|u'_m-u\|_{SBV(\Omega)}=0.
\end{equation*}
\end{enumerate}
\end{proposition}

On the other hand, these properties of $u'_m$ yet do not entitle us to apply the energy control Lemma \ref{limsup_regular} directly, since it requires \textit{boundedness of gradient} as well. Therefore, we need to perform some sort of surgery on $u'_m$ near $\Gamma_m$. 

\begin{proposition}
Suppose $\{u'_m\}_{m=1}^{\infty}$ is defined as above, which, in particular, satisfies $\|u'_m-u\|_{SBV(\Omega)}\to0$, there exists another sequence $\{u''_m\}_{m=1}^{\infty}$ converging to the same limit $u\in SBV(\Omega)$, which not only possesses all the properties specified as in Proposition \ref{sequence_properties}, but is also subject to the element-wise bound $\|\nabla^a u''_m\|_{L^{\infty}}<\infty$ ($m\geq1$).
\end{proposition}

\begin{proof}
{Fundamentally, our approach is by cutting off a narrow area near the singularity of $u$, which potentially consists of an unbounded (though pointwise finite) gradient, and then growing the nearby landscape of the same function into that vicinity.}

\textbf{Step 1.} {(Defining the tubular neighbourhoods.)}
Recall that the singularity upper-bound of $u'_m$ is expressed as $\Gamma_m=\bigcup_{i=1}^{N_m}\Gamma_{m,i}$ where each $\Gamma_{m,i}$ is a closed $C^1-$curve without self-intersection, it suffices to conduct the surgery on each component $\Gamma_{m,i}$ sequentially. 

By classical tubular neighborhood theorem, there exists a closed tubular neighbourhood $D_{m,i}(\epsilon^0_m)$ of $\Gamma_{m,i}$, together with a $C^1-$bijection $$\psi_{m,i}:[0,a_m]\times[-\epsilon^0_m,\epsilon^0_m]\to D_{m,i}(\epsilon^0_m),$$ such that $\psi_{m,i}([\frac{a_m}{4},\frac{3a_m}{4}]\times\{0\})\supset \Gamma_{m,i}$. We can also choose each $\epsilon^0_m$ small enough such that $D_{m,i}(\epsilon^0_m)\cap D_{m,j}(\epsilon^0_m)=\emptyset$ for any $i\neq j$. 

It immediately follows that the Jacobian $J(\psi_{m,i})$ (denoted simply as $J$ if the context is clear), its determinant, and its inverse are all continuous on this corresponding closed domain. As a result, there exists a constant $c=c_m>1$ such that the boundedness conditions $$|J(s,t)_{k,l}|<c, |J^{-1}(x)_{k,l}|<c, \|J\|_{2,2}<c \text{ and } \det(J(s,t))\in[c^{-1},c]$$ are satisfied for all $x\in D_{m,i}(\epsilon^0_m)$ and $(s,t)\in[0,a_m]\times[-\epsilon^0_m,\epsilon^0_m]$.

 We fix $\epsilon_m\leq\epsilon^0_m$ such that $\int_{D_{m,i}(\epsilon)}|\nabla^au'_m|<(c_mN_m)^{-1}2^{-m}$ for all $i=1,\dots,N_m$ (thus their sum is controlled by $2^{-m}$) and denote the closed domain simply as $D_{m,i}$ if the context is clear. 

{\textbf{Step 2.} (Trace difference control.)} We consider the truncated function $\bar u_m:=\max\{\min\{u'_m,B_m\},-B_m\} \in SBV(\Omega)$ where the number $B_m>0$ is chosen such that $\|\bar u_m-u'_m\|_{BV}<2^{-m}$. Note that this function is continuous both on $D^+_{m,i}:=\psi_{m,i}([0,a_m]\times[0,\epsilon_m])$ and on $D^-_{m,i}:=\psi_{m,i}([0,a_m]\times[-\epsilon_m,0])$.

Consider the positive part: denote the trace as $\bar u_{m,i}^+$, we have $$\lim_{t\to0^+}\bar u(\psi_{m,i}(s,t))=\bar u_{m,i}^+(\psi_{m,i}(s,0)).$$ Note that $\bar u_m\in SBV(\Omega)$ implies $\bar u^+_{m,i}\in L^1(\Gamma_{m,i})$. By Egorov's theorem, for any $\eta>0$, there exists $t_\eta>0$ such that for any $t\leq t_\eta$,
\begin{equation}
\mathcal H^1(\{s\in[0,a_m]:|\bar u_m(\psi_{m,i}(s,t))-\bar u_{m,i}^+(s)|\geq\eta\})<\eta.
\end{equation}
In particular, if $\eta\leq\eta_1:=\min\left\{\frac{2^{-m}}{2a_m\cdot N_m},\frac{2^{-m}}{2B_m\cdot N_m}\right\}$, {we have estimated between the (one-side) trace and the $t$-translated counterpart as follows,}
\begin{eqnarray}
&&\int_{\Gamma_{m,i}}|\bar u_m(s(x),t_\eta)-\bar u^+_{m,i}(x)|dx\\
&\leq&c_m\cdot\int_0^{a_m}|\bar u_m(\psi_{m,i}(s,t_\eta))-\bar u_{m,i}^+(\psi_{m,i}(s,0))|ds\\
&\leq&N_m^{-1}2^{-m},
\end{eqnarray}
{where the constant $c_m$ comes from the apparent coordinate transform (\emph{i.e.} the parametrization).}

\textbf{Step 3.} {(Domain deformation.)}
Consider a smooth function $\gamma:[0,1]\to\mathbb R^+$ such that $\gamma(s)=1$ for $s\in[\frac{1}{4},\frac{3}{4}]$ and $\gamma(s)=0$ for $s\in[0,\frac{1}{8}]\cup[\frac{7}{8},1]$, and accordingly the domain
\begin{equation}
\tilde D^+_{m,i}=\{\psi_{m,i}(s,t):s\in[0,a_m],\ t\in[t_\eta\gamma(s/a_m),\epsilon_m]\}\subset D^+_{m,i}.
\end{equation}
We define a bijective deformation $\phi^+_{m,i}:\tilde D^+_{m,i}\to D^+_{m,i}$ as
\begin{equation}
\phi^+_{m,i}(x)=\psi_{m,i}\left(s(x),\frac{\epsilon_m(t(x)-t_\eta\gamma(s(x)/a_m))}{\epsilon_m-t_\eta\gamma(s(x)/a_m)}\right),
\end{equation}
where $(s(x),t(x))=\psi^{-1}_{m,i}(x)$ are the coordinates of the point $x$. We notice that:

\begin{enumerate}
    \item[(i)] For boundary point $x\in[0,a_m]\times\{\epsilon_m\}\cup\{0,a_m\}\times[0,\epsilon_m]$ it has $\phi^+_{m,i}(x)=x$;
    \item[(ii)] The $t=t_\eta$ curve is pushed down to the $t=0$ curve $\Gamma_{m,i}$ where $s\in[\frac{a_m}{4},\frac{3a_m}{4}]$, while it goes trivial (i.e. there is no deformation) where $s\in[0,\frac{a_m}{8}]\cup[\frac{7a_m}{8},a_m]$.
\end{enumerate} 
\begin{figure}
    \centering
    \includegraphics[scale=0.75]{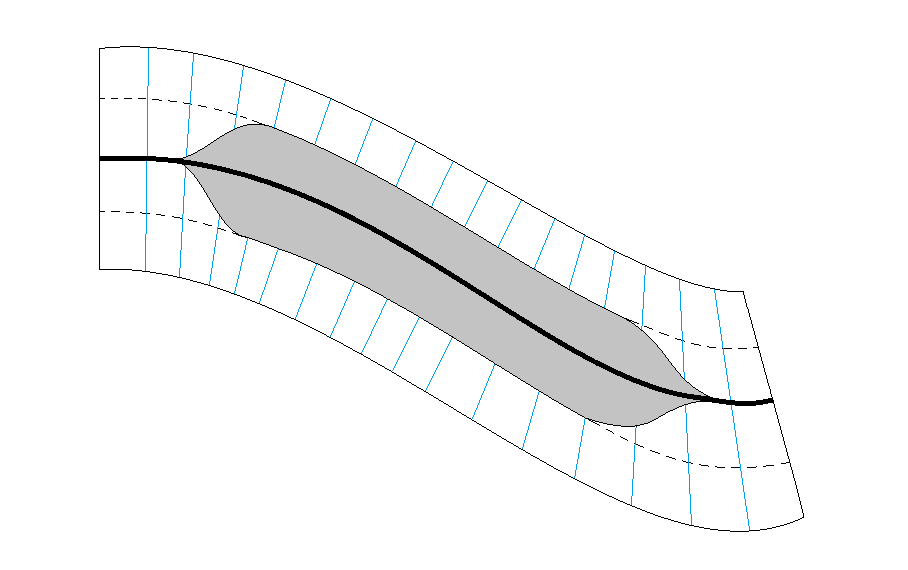}
    \caption{Deformations $\phi_{m,i}^{\pm}$, through which the function value in the shaded areas are excluded (neglected).}
    \label{domain_deform}
\end{figure}

For the domain $D^-_{m,i}$, we define $\tilde D^-_{m,i}$ and $\phi^-_{m,i}$ in the same way. Finally, we define $u''_m$ on $\Omega$ as
\begin{equation}
u''_m(x)=\left\{
\begin{matrix}
\bar u_m(x), & if\ x\in\Omega\setminus \bigcup_{i=1}^{N_m} D_{m,i}\\
\bar u_m((\phi^\pm_{m,i})^{-1}(x)), & if\ x\in D^\pm_{m,i}\\
\end{matrix}
\right.
\end{equation}
and verify that: (i) Jump of $u''_m$ is also supported on $\Gamma_m$, and is about $2^{-m}$ close to that of $u'_m$; (ii) $\int_{D_{m,i}}|\nabla^au''_m|<2\cdot2^{-m}$, that is, near singularity variation is at most doubled by our deformation; (iii) Bounded $\nabla^au''_m$ as we hope.

\end{proof}

The proof of Proposition~\ref{limsup} is a direct consequence of the $L^1$-convergence and the following proposition. 

\begin{proposition}\label{prop:nonhomo-moll}
    For any function $u\in SBV(\Omega)$, there exists a sequence $\{u^*_n\}_{n=1}^\infty\subset SBV(\Omega)$ such that $\|u^*_n-u\|_{SBV(\Omega)}\to0$ and $\lim_{n\to\infty}R_n(u^*_n)\geq R(u)$.
    \end{proposition}
    \begin{proof}
    Now we arrange the sequence $\{u''_m\}_{m=1}^\infty$ defined in \eqref{eqn_conv} into what we need. As is indicated in \Cref{lemma_conv_sbv}, for any fixed $u''_m$ its non-trvial jump set $\Gamma^*(u''_m)\subset\Gamma_m$ where $\Gamma_m$ is finite and closed, starting from \Cref{limsup_regular} we have
    \begin{eqnarray*}
\limsup_{n\to\infty}R_n(u''_m)&\leq&\int_\Omega|\nabla^au''_m|+\alpha\cdot\mathcal H^1(\Gamma_m)\\
&\leq&\left(\int_\Omega|\nabla^au|+\|u''_m-u\|_{SBV(\Omega)}\right)+\alpha\cdot\left(\mathcal H^1(\Gamma^*)+2^{-m}\right)\\
    &=&R(u)+\beta_m,
    \end{eqnarray*}
    where $\beta_m:=\|u'_m-u\|_{SBV(\Omega)}+2^{-m}>0$ and $\beta_m\to0$. Define the number $N'_m\in\mathbb Z^+$ such that $E_n(u''_m)\leq E(u)+2\beta_m$ for any $n\geq N'_m$, then define $u''_0:=0$, $N'_0:=0$, $N_m:=m+N'_m$, and
    \begin{equation*}
    u^*_n:=u''_m\ \ \text{if}\ \ N_m+1\leq n\leq N_{m+1},
    \end{equation*}
    for each $m\in\mathbb Z^+$. The sequence $\{u^*_n\}_{n=1}^\infty$ satisfies desired conditions.
    \end{proof}

\section*{Acknowledgement}

Bin Dong is supported in part by the New Cornerstone Investigator Program 

\appendix

\section{{Counting lemmas on the lattice}}
\label{sec_lemmas}

{Here we summarize a couple of lemmas concerning the grid geometry around the $C^1$-curve $\Gamma^*_u$, where the function $u$ exhibits jump behaviour.}

\robustintersection*
\begin{proof}[Proof of Lemma~\ref{robust_intersection}]

Without loss of generality, we assume that $\Gamma = \cup_{i = 1}^K \Gamma_i$ and $n$ is chosen sufficiently small such that the length of $\Gamma_i$ are larger then $2^{-n+1}$. Suppose that $x \in \square_{n,k} \cap \Gamma_i$ for some $\Gamma_i$. Define $\square^1_{n,k}:=\bigcup_{l\in\left\{-1,0,1\right\}^2}\square_{n,k+l}$. Suppose that $\Gamma_i$ is entirely inside $\square^1_{n,k}$, then the lemma holds obviously. Suppose not, the set $P:=\{s\in[t,t+4\cdot2^{-n}]:\gamma_i(s)\in\partial\square^1_{n,k}\}\neq\emptyset$. By continuity of $\gamma_i$ and closedness of the boundary in question, the set $P$ is closed. Define $t':=\inf_{s\in P}s$ and $y:=\gamma_i(t')\in\partial\square^1_{n,k}$. Since $x\in\square_{n,k}$, we have $|x-y|\geq2^{-n}$ and the open segment $\gamma_i((t,t'))$ has curve length $|\gamma_i((t,t'))|_1\geq2^{-n}$. Simultaneously, we note that $\gamma_i((t,t'))\cap\overline{\left(\square^1_{n,k}\right)^c}=\emptyset$ by minimality of $t'$. Therefore we have
    \begin{eqnarray*}
    &&\sum_{l\in\{-1,0,1\}^2}\left|\gamma_i((t,t'))\cap\square_{n,k+l}\right|_1\\
    &=&\left|\gamma_i((t,t'))\cap\square^1_{n,k}\right|_1\\
    &=&|\gamma_i((t,t'))|_1\\
    &\geq&2^{-n},
    \end{eqnarray*}

\end{proof}

    \boundintersection*
    \begin{proof}
    Following the construction applied in the proof of Lemma \ref{robust_intersection}, again, we resort to taking sub-intervals when necessary, and assume $\Gamma=\bigcup_{i=1}^K\gamma_i([a_i,b_i])$ (in which $b_i-a_i>0$), where any $s,t\in[a_i,b_i]$ have $|\dot\gamma_i(s)-\dot\gamma_i(t)|<\pi/4$. For any $\gamma_i(t)\in\square_{n,k}$, as long as $|s-t|\geq2\cdot2^{-n}$, we have $|(\gamma_i(s)-\gamma_i(t))\cdot\dot\gamma_i(t)|\geq\sqrt2\cdot2^{-n}$ and thus $\gamma_i(s)\notin\square_{n,k}$. Therefore
    \begin{equation*}
|\gamma_i([a_i,b_i])\cap\square_{n,k}|_1\leq|\gamma_i([\max\{a_i,t-2\cdot2^{-n}\},\min\{b_i,t+2\cdot2^{-n}\}])|_1\leq4\cdot2^{-n},
    \end{equation*}
    and $|\Gamma\cap\square_{n,k}|_1\leq4\cdot K\cdot2^{-n}$ for any $n\geq1$.
    \end{proof}
    Now we are ready to prove \Cref{area_convergence}, which is the core lemma of these box counting lemmas.

    \begin{proof}[Proof of Lemma~\ref{area_convergence}]
    \begin{figure}
        \centering
        \includegraphics[scale=0.75]{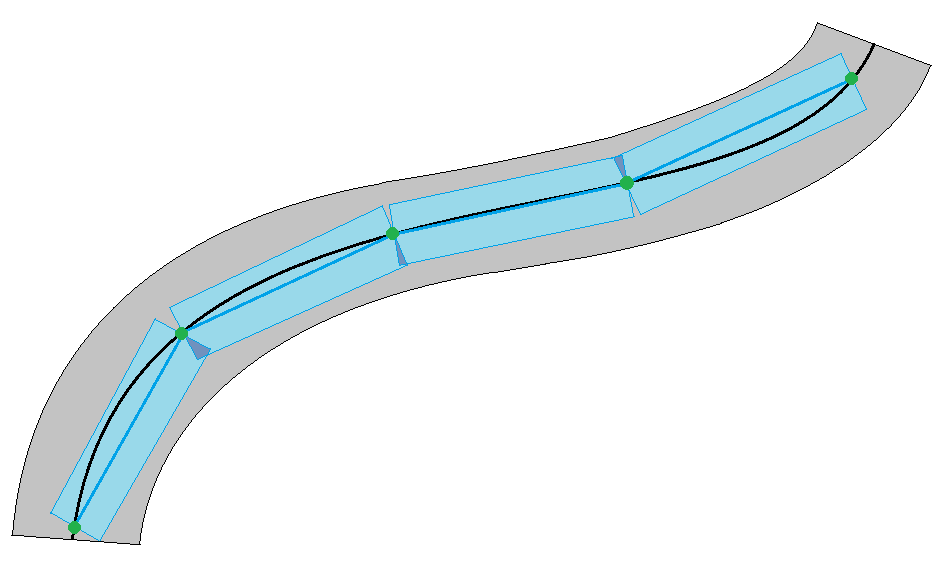}
        \caption{Approximating the $H_n$-tubular neighbourhood by a chain of rectangles of width $H'_n$.}
        \label{fig:tubular_rectangular}
    \end{figure}
    We first assume $\Gamma=\gamma([0,b])$, where $\gamma\in C^1([0,b])$ (thus $\dot\gamma\in C^0([0,b])$ and is uniformly continuous on $[0,b]$). For each $\eta\in(0,b]$, define
    \begin{equation}
    \mu(\eta):=\sup_{t,t'\in[0,b],\ |t-t'|\leq\eta}\left|\dot\gamma(t)-\dot\gamma(t')\right|.
    \end{equation}
    It is clear that $\lim_{\eta\to0^+}\mu(\eta)=0$. By simple geometric calculations, we get constant $C_1=\frac{1}{8}>0$ such that
    \begin{equation}
    |\gamma(t)-\gamma(t')|\geq\left(1-C_1\cdot\mu(\eta)^2\right)\cdot|t-t'|
    \end{equation}
    for any $|t-t'|\leq\eta$. Define $\tilde\Gamma_n$ as the piecewise linear curve that connects points
    \begin{equation}
    \gamma(0),\gamma(H_n),\cdots,\gamma((\lceil b/H_n\rceil-1)\cdot H_n),\gamma(b)
    \end{equation}
    in consequence. Denote $\mu_n:=\mu({4}H_n)$, $z_{n,i}:=\gamma(i\cdot{4}H_n)$ for each $i=0,\cdots,\lceil b/({4}H_n)\rceil-1$ and $z_{n,\lceil b/({4}H_n)\rceil}:=\gamma(b)$, we have
    \begin{equation}
    \arccos\left(\frac{z_{n,i+1}-z_{n,i}}{|z_{n,i+1}-z_{n,i}|}\cdot\frac{z_{n,i}-z_{n,i-1}}{|z_{n,i}-z_{n,i-1}|}\right)\leq2\mu_n
    \end{equation}
    for relevant indices $i$. For each (sufficiently large) $n$ such that $\mu_n\leq\pi/6$, we define an invertible, contractive projection $\Pi_n:\Gamma\to\tilde\Gamma_n$ such that $|x-\Pi_n(x)|\leq C_2H_n\mu_n$. In particular, for any $x=\gamma(t)$ where $t\in[0,b)$, there exists precisely one index $i$ such that $t\in[iH_n,\min\{(i+1)H_n,b\})$, we define 
    \begin{equation}
    \Pi_n(x)=z_{n,i}+\frac{\left<x-z_{n,i},z_{n,i+1}-z_{n,i}\right>}{|z_{n,i+1}-z_{n,i}|^2}\cdot(z_{n,i+1}-z_{n,i}),
    \end{equation} 
    and also $\Pi_n(\gamma(b)):=\gamma(b)$.

    Denote $\tilde S_n:=\Pi_n(S_n)$, we have $\mathcal H^1(\tilde S_n)\leq\mathcal H^1(S_n)$. For any $x\in\Omega$, such that $dist(x,\tilde\Gamma_n\setminus\tilde S_n)\leq H'_n:=H_n(1-C_2\mu_n)$, there exists $y'\in\tilde\Gamma_n\setminus\tilde S_n$ such that $|x-y'|\leq H_n(1-C_2\mu_n)+\epsilon$ for each $\epsilon>0$. Note that $y:=\Pi_n^{-1}(y')\in\Gamma\setminus S_n$ and $|x-y|\leq|x-y'|+|y'-y|\leq H_n+\epsilon$, therefore $dist(x,\Gamma\setminus S_n)\leq H_n$. That is,
    \begin{equation}
    \{x\in\Omega:dist(x,\tilde\Gamma_n\setminus\tilde S_n)\leq H'_n\}\subset\{x\in\Omega:dist(x,\Gamma\setminus S_n)\leq H_n\}.
    \end{equation}
    At any point $x\in\tilde\Gamma'_n:=\tilde\Gamma_n\setminus\{z_{n,i}\}_{1\leq i\leq\lceil\frac{b}{2H_n}\rceil}$, its unit $\tilde\Gamma_n-$tangent vector is well-defined, whose expression, in particular, can be written as $\tau_x=\frac{z_{n,i+1}-z_{n,i}}{|z_{n,i+1}-z_{n,i}|}$ where $i$ denotes the index of segment on which $x$ stays. Define at this $x\in\tilde\Gamma'_n$ a line segment
    \begin{equation}
    J_{n,x}:=\{y\in\Omega:(y-x)\cdot\tau_x=0,\ |y-x|\leq H'_n\},
    \end{equation}
    and we have
    \begin{equation}
    \bigcup_{x\in\tilde\Gamma'_n\setminus\tilde S_n}J_{n,x}\subset\{x\in\Omega:dist(x,\tilde\Gamma_n\setminus\tilde S_n)\leq H'_n\},
    \end{equation}
    where the left hand side is actually a finite union of rectangles that each overlaps at most with its two neighbours if $n$ is sufficiently large (since locally we have $(1-C_1\cdot\mu_n^2)\cdot4H_n\geq2H_n$, and globally, the non-intersecting condition. Each of these quadragonal overlaps apparently has an area upper-bound $(H_n')^2\tan\mu_n$.
    
    Altogether, we attain the estimate
    \begin{eqnarray}
    &&(2H_n)^{-1}\cdot\mathcal H^2(\{x\in\Omega:dist(x,\Gamma\setminus S_n)\leq H_n\})\\
    &\geq&(2H_n)^{-1}\cdot\mathcal H^2(\{x\in\Omega:dist(x,\tilde\Gamma_n\setminus\tilde S_n)\leq H'_n\})\\
    &\geq&(2H_n)^{-1}\cdot\mathcal H^2\left(\bigcup_{x\in\tilde\Gamma'_n\setminus\tilde S_n}J_{n,x}\right)\\
    &\geq&(2H_n)^{-1}\cdot\left(\left(b(1-C_1\mu_n^2)-\delta\right)\cdot2H'_n-\lceil b/({4}H_n)\rceil\cdot (H'_n)^2\tan\mu_n\right)\\
    &\geq&b\cdot\left(1-C_1\mu_n^2-\left({\frac{1}{2}}+C_2\right)\mu_n\right)-\delta,
    \end{eqnarray}
    and conclude the proof immediately.
    \end{proof}

\section{Technical details of the mollifier}\label{proof_sequence_properties}
In the rest of this section we validate the properties claimed in Proposition \ref{sequence_properties} each as a separate lemma. Symbols should denote the same objects (operations) as are defined previously, if not otherwise specified.

\begin{lemma}\label{bv_compose}
Assume $u\in BV(\Omega)$, and $f\in C^2(\Omega';\Omega)$ is an invertible transform. If there exists a constant $R>0$ such that at each $x\in\Omega'$, $y\in\Omega$, the Jacobian matrices $J_f,J_{f^{-1}}$ have norm bound $\|J_f(x)\|_2,\|J_{f^{-1}}(y)\|_2\leq \kappa$ as well as parity $\det(J_f(x))>0$, we have $u\circ f\in BV(\Omega')$, moreover, $|u\circ f|_{BV(\Omega')}\in[\kappa^{-1}|u|_{BV(\Omega)},\kappa|u|_{BV(\Omega)}]$.
\end{lemma}
\begin{proof}
Without loss of generality, we assume $u\neq0$. By definition of total variation, there exists a vector field $v^\epsilon\in C^\infty(\Omega;\mathbb R^2)$ with $\sup_{x\in\Omega}|v^\epsilon(x)|\leq1$, such that
\begin{equation}
-\int_\Omega u(x)(\nabla\cdot v^\epsilon)(x)dx\geq|u|_{BV(\Omega)}-\epsilon.
\end{equation}
We define
\begin{equation}
w^\epsilon(x)=v^\epsilon(f(x))\cdot \kappa^{-1}
\left(
\begin{matrix}
\frac{\partial f_2}{\partial x_2} & -\frac{\partial f_2}{\partial x_1}\\
-\frac{\partial f_1}{\partial x_2} & \frac{\partial f_1}{\partial x_1}
\end{matrix}
\right),
\end{equation}
and verify that $|w^\epsilon(x)|\leq|v^\epsilon(f(x))|\leq1$ for any $x\in\Omega$. Moreover,
\begin{eqnarray}
(\nabla\cdot w^\epsilon)(x)&=& \kappa^{-1}(\nabla\cdot v^\epsilon)(f(x))\cdot\left(\frac{\partial f_1}{\partial x_1}\cdot\frac{\partial f_2}{\partial x_2}-\frac{\partial f_2}{\partial x_1}\cdot\frac{\partial f_1}{\partial x_2}\right)\\
&=& \kappa^{-1}(\nabla\cdot v^\epsilon)(f(x))\cdot\det(J_f(x)),
\end{eqnarray}

and consequently
\begin{eqnarray}
&&-\int_{\Omega'}(u\circ f)(x)\ (\nabla\cdot w^\epsilon)(x)dx\\
&=&-\kappa^{-1}\int_{\Omega'} u(f(x))\ (\nabla\cdot v^\epsilon)(f(x))\cdot\det(J_f(x))dx\\
&=&-\kappa^{-1}\int_\Omega u(y)\ (\nabla\cdot v^\epsilon)(y)dy\\
&\geq&\kappa^{-1}(|u|_{BV(\Omega)}-\epsilon).
\end{eqnarray}
We obtain $|u\circ f|_{BV(\Omega')}\geq \kappa^{-1}(|u|_{BV(\Omega)}-\epsilon)$ by taking supremum, and conclude that $|u\circ f|_{BV(\Omega')}\geq R^{-1}|u|_{BV(\Omega)}$, since $\epsilon>0$ is arbitrarily small. Following the same logic but with $u\circ f$, $f^{-1}$ instead of $u$, $f$ involved in the argument, for any $a>0$ such that $|u\circ f|_{BV(\Omega')}\geq a$, we have $|u|_{BV(\Omega)}\geq R^{-1}(a-\epsilon)$. Therefore $|u\circ f|_{BV(\Omega')}<\infty$, that is, $u\circ f\in BV(\Omega')$. Moreover, $|u|_{BV(\Omega)}\geq R^{-1}|u\circ f|_{BV(\Omega')}$, and we conclude the lemma.
\end{proof}

\begin{lemma}\label{lemma_conv_type}
For any $f\in L^1(\Omega)$, we have $\lim_{m\to\infty}\|f*'_m\eta-f\|_{L^1(\Omega)}=0$.
\end{lemma}
\begin{proof}
Since $f\in L^1(\Omega)$, there exists an $\tilde f\in C^\infty(\overline\Omega)$ such that $\|\tilde f-f\|_{L^1(\Omega)}<\epsilon$ for an arbitrarily prescribed precision $\epsilon>0$. (We denote the restriction of $\tilde f$ to $\Omega$ still as $\tilde f$.) Denote $\tilde f'_m:=\tilde f*'_m\eta$, we have
\begin{eqnarray*}
&&\int_{\Omega}|\tilde f'_m(x)-f'_m(x)|dx\\
&\leq&\int\int|\tilde f(x-r_m(x)y)-f(x-r_m(x)y)|\cdot\eta(y)dydx\\
&=&\int\left(\int|\tilde f(x-r_m(x)y)-f(x-r_m(x)y)|dx\right)\eta(y)dy\\
&\leq&\int4\|\tilde f-f\|_{L^1(\Omega)}\eta(y)dy\\
&=&4\epsilon
\end{eqnarray*}
for any $m\geq1$. By construction, $\tilde f$ is uniformly continuous on $\overline\Omega$, thus for sufficiently large $m$, $|x-x'|\leq 2^{-m}$ implies $|\tilde f(x)-\tilde f(x')|<\epsilon$. Consequently,
\begin{eqnarray*}
\int_\Omega|\tilde f'_m(x)-\tilde f(x)|dx&\leq&\int_\Omega\int|\tilde f(x)-\tilde f(x-r_m(x)y)|\cdot\eta(y)dydx\\
&\leq&\int_\Omega\int\epsilon\cdot\eta(y)dydx\\
&\leq&\epsilon.
\end{eqnarray*}
Altogether, we have
\begin{equation*}
\|f'_m-f\|_{L^1}\leq\|f'_m-\tilde f'_m\|_{L^1}+\|\tilde f'_m-\tilde f\|_{L^1}+\|\tilde f-f\|_{L^1}\leq6\epsilon
\end{equation*}
as is desired.
\end{proof}

\begin{lemma}\label{lemma_jump}
    The jump distribution of $u'_m$ is equal to that of $u$, that is $\nabla^s u'_m=\nabla^s u$, on $\Gamma_m$.
    \end{lemma}
    \begin{proof}
    Throughout the proof of this lemma the index $m$ is fixed, thus we treat it as a constant. By Proposition \ref{fine_property}, for $\mathcal H^1-$a.e. $x\in\Gamma_m$,
    \begin{equation}\label{limit_local}
    \lim_{r\to0}\int_{B(0,1)}|u(x+ry)-u^+_x \chi_{y \cdot \nu > 0} + u^-_x \chi_{y \cdot \nu < 0}|^2dy=0,
    \end{equation}
    and the above equation yields $L^1$-convergence.

     Note that end points of $\Gamma_m$ are finite and therefore possess zero $\mathcal H^1-$measure. For $\mathcal H^1-$a.e. $x$ with limit \eqref{limit_local}, the local segment of $\Gamma_m$ around $x$ splits $B(x,2r)$ into two disjoint parts, denoted $\tilde B^+(x,2r)$ and $\tilde B^-(x,2r)$ respectively, as far as $r$ is sufficiently small (for instance, when $r\leq r_0/2$). Let the curve step function $\tilde\chi_x$ be defined as
     \begin{equation}
     \tilde\chi_x(y)=\left\{
     \begin{matrix}
     a^+ & \text{if}\ \ y\in\tilde B^+(x,r_0)\\
     a^- & \text{if}\ \ y\in\tilde B^-(x,r_0).
     \end{matrix}
     \right.
     \end{equation}
    Let $x=\gamma_m(t)$. Note the curve $\Gamma_m$ has $C^1$ regularity. Without loss of generality, we assume that it is parametrized by curve length, that is, $|\dot\gamma_m(t)|=1$ for all relevant $t$. For any $\epsilon\in(0,1/4)$, there exists $r_1>0$ such that $|\dot\gamma_m(s)-\dot\gamma_m(t)|<\frac{2}{\pi}\cdot\epsilon$ as long as $|s-t|\leq r_1$. Then for any $r\leq \frac{1}{\sqrt2}\cdot r_1$, for any $x'\in\Gamma_m\cap B(x,r)$, the line segment $x'x$ are included in the (two-side) $\pm\epsilon$-cone of the tangent line at $x$, we have
     \begin{eqnarray}\label{tilde_eqn}
     &&\lim_{r\to0}r^{-2}\int_{B(x,r)}|\tilde\chi_x(y)-\chi_x(y)|dy\\
     &\leq&\lim_{r\to0}r^{-2}\left(\frac{4\epsilon}{2\pi}\cdot\pi r^2\right)\\
     &=&2\epsilon.
     \end{eqnarray}
     Since $\epsilon$ can be taken arbitrarily small, this limit must be equal to 0. Then we have
     \begin{equation}
     \lim_{r\to0}r^{-2}\int_{\tilde B^\pm(x,r)}|u(x)-a^\pm|dx=0.
     \end{equation}
     Concerning the smoothened function $u'_m$, (for sufficiently large $m$) we have
     \begin{eqnarray}
     &&r^{-2}\int_{\tilde B^+(x,r)}|u'_m(y)-a^+|dy\\
     &=&r^{-2}\int_{\tilde B^+(x,r)}\left|\int_{B(0,1)}(u(y-r_m(y)z)-a^+)\eta(z)dz\right|dy\\
     &\leq&r^{-2}\int_{B(0,1)}\left(\int_{\tilde B^+(x,r)}|u(y-r_m(y)z)-a^+|dy\right)\cdot\eta(z)dz\\
     &\leq&r^{-2}\int_{B(0,1)}\left(2\int_{\tilde B^+(x,2r)}|u(y')-a^+|dy'\right)\cdot\eta(z)dz\\
     &=&2\cdot r^{-2}\int_{\tilde B^+(x,2r)}|u(y')-a^+|dy'\to0,
     \end{eqnarray}
     in which a coordinate change $y'=\tau(y):=y-r_m(y)x$ has been applied, where
     \begin{eqnarray}
     |y'-x|&\leq&|y'-y|+|y-x|\\
     &\leq&|r_m(y)|\cdot|z|+|y-x|\\
     &<&2r,
     \end{eqnarray}
     and, moreover, $y'\in\tilde B^+(x,2r)$. Indeed, if $y'\notin \tilde B^+(x,2r)$, then the line segment $\overline{y'y}$ has at least one intersection $y''$ with $\Gamma_m$ and $|y''-y'|<r_m$, which contradicts the fact $dist(y,\Gamma_m)\geq r_m$. Moreover, its Jacobian determinant $|\det J_\tau|\leq2$ for any sufficiently large $m$. 
     
     In the same way we get $r^{-2}\int_{\tilde B^-(x,r)}|u'_m(y)-a^-|dy\to0$. Add them together and enforce Equation \eqref{tilde_eqn}, we obtain the convergence
     \begin{equation}
     \lim_{r\to0}r^{-2}\int_{B(x,r)}|u'_m(y)-\tilde\chi_x(y)|dy=0.
     \end{equation}
    
     It is clear through the proof of Lemma \ref{lemma_conv_sbv} that $\nabla u'_m$ is absolutely continuous as a measure on $\Omega\setminus\Gamma_m$, that is, $u'_m$ may have jump discontinuity only on (a subset of) $\Gamma_m$.  By structure theorem of functions of bounded variation, for $\mathcal H^1-$a.e. $x\in\Omega$, there exists some step function $\chi'_x$ such that 
     \begin{equation}
     \lim_{r\to0}r^{-2}\int_{B(x,r)}|u'_m(y)-\chi'_x(y)|^2dy=0,
     \end{equation}
     or equivalently, $\lim_{r\to0}\|u'_m((\cdot-x)/r)-\chi'_x(\cdot-x)\|_{L^2(B(0,1))}=0$, and consequently the convergence $u'_m((\cdot-x)/r)\to\chi'_x(\cdot-x)$ also occurs in $L^1(B(0,1))$, and we identify $\chi'_x=\tilde\chi_x\in L^1(B(0,1))$ for $\mathcal H^1-$a.e. $x\in\Gamma_m$ and conclude our proof. 

    \end{proof}

\begin{lemma}\label{lemma_conv_sbv}
For each $m\in\mathbb Z^+$, $u'_m\in SBV(\Omega)$. In particular, its non-trival jump set $\Gamma^*(u'_m)\subset\Gamma_m$, equivalently, $\nabla u'_m=\nabla^a u'_m$ on $\Omega\setminus\Gamma_m$.
\end{lemma}

\begin{proof}
By change of coordinates, it is easy to see that
\begin{eqnarray*}
u'_m(x)&:=&\int u(x-r_m(x)y)\eta(y)dy\\
&=&\int u(y)\cdot\frac{1}{r_m(x)^2}\eta\left(\frac{x-y}{r_m(x)}\right)dy.
\end{eqnarray*}

Since $r_m\in C^1(\Omega\setminus\Gamma_m)$, we immediately obtain that $u_m'(x)\in C^1(\Omega\setminus\Gamma_m)$. In other words, $\nabla u'_m=\nabla^au'_m$ as measures on $\Omega\setminus\Gamma_m$. We regard $u(x-r_m(x)y)$ as a function of $x$ on $\Omega\setminus\Gamma_n$ (parametrized by $y\in B(0;1)$, $m\geq1$), and apply Lemma \ref{bv_compose} to $u\in BV(\Omega)$ and $f(x)=x-r_m(x)y:\Omega\setminus\Gamma_m\to\Omega\setminus\Gamma_m$ (where its dependence on $y$ and $m$ is understood). We verify that $f\in C^{\infty}(\Omega\setminus\Gamma_m)$ and
\begin{equation*}
J_f(x)=\left(
\begin{matrix}
1 & 0\\
0 & 1
\end{matrix}
\right)-\left(
\begin{matrix}
y_1\cdot\partial_1r_m(x) & y_1\cdot\partial_2r_m(x)\\
y_2\cdot\partial_1r_m(x) & y_2\cdot\partial_2r_m(x)
\end{matrix}
\right)
\end{equation*}
where $|y_i\cdot\partial_jr_m(x)|\leq|y|\cdot|\nabla r_m(x)|\leq 2^{-m}$ by construction. Consequently, $\|J_{f}(x)\|_2\leq2$, $\|J_{f^{-1}}(x)\|_2\leq2$ for any $y\in B(0;1)$ and any sufficiently large $m$, and we have
\begin{equation*}
\int_\Omega|\nabla_x u(x-r_m(x)y)|dx\leq 2\int_\Omega|\nabla u(x)|dx.
\end{equation*}

The total variation of $u'_m$ on $\Omega\setminus\Gamma_m$ can therefore be estimated as
\begin{eqnarray*}
\int_{\Omega\setminus\Gamma_m}|\nabla u'_m(x)|dx&=&\int_{\Omega\setminus\Gamma_m}\left|\nabla_x \int u(x-r_m(x)y)\eta(y)dy\right|dx\\
&\leq&\int\left(\int_{\Omega\setminus\Gamma_m}|\nabla_x u(x-r_m(x)y)|dx\right)\cdot\eta(y)dy\\
&\leq&\int\left(2\int_{\Omega\setminus\Gamma_m}|\nabla u(x)|dx\right)\eta(y)dy\\
&=&2\int_{\Omega\setminus\Gamma_m}|\nabla u(x)|dx<\infty.
\end{eqnarray*}
Now we consider the variation of $u'_m$ on the entire domain $\Omega$. In particular, we take a sufficiently small tubular neighbourhood $\tilde\Gamma_m$ of $\Gamma_m$, and note that for any $w\in C^\infty_0(\Omega)$ such that $|w(x)|\leq1$, we have a decomposition $w(x)=v(x)+\sigma(x)$ where $v$ and $\sigma$, having $|v(x)|<1$, $|\sigma(x)|<1$, are supported on $\Omega\setminus\Gamma_m$ and $\tilde\Gamma_m$ respectively. In this way,
\begin{eqnarray*}
&&\int_\Omega u'_m(x)(\nabla\cdot w)(x)dx\\
&=&\int_{\Omega\setminus\Gamma_m}u'_m(x)(\nabla\cdot v)(x)dx+\int_{\tilde\Gamma_m}u'_m(x)(\nabla\cdot\sigma)(x)dx\\
&=&\int_{\Omega\setminus\Gamma_m}u'_m(x)(\nabla\cdot v)(x)dx\\
&&\ \ \ \ +\left(-\int_{\tilde\Gamma_m\setminus\Gamma_m}\left<\nabla u(x),\sigma(x)\right>dx+\int_{\Gamma_m}(tr^+[u'_m](x)-tr^-[u'_m](x))\left<\sigma(x),n(x)\right>dx\right)\\
&\leq&\int_{[\Omega\setminus\Gamma_m]}|\nabla u'_m(x)|dx+\int_{\tilde\Gamma_m\setminus\Gamma_m}|\nabla u'_m(x)|dx+\int_{\Gamma_m}|tr^+[u'_m](x)-tr^-[u'_m](x)|dx
\end{eqnarray*}

That is, $u'_m\in BV(\Omega)$. Furthermore, since $\nabla u'_m=\nabla^au'_m$ on $\Omega\setminus\Gamma_m$, the Cantor part of this gradient, $\nabla^cu'_m$, vanishes on $\Omega\setminus\Gamma_m$, and therefore vanishes on $\Omega$ since $\mathcal H^1(\Gamma_m)<\infty$ (which implies $\mathcal H^1(A)<\infty$, thus $\nabla^su'_m(A)=0$, for any $A\subset\Gamma_m$), that is, $u'_m\in SBV(\Omega)$. Moreover, $\Gamma^*(u'_m)\subset\Gamma_m$.

\end{proof}

\begin{lemma}
Let $u$, $u'_m$ be defined as above. There is convergence
\begin{equation*}
\lim_{m\to\infty}\|u'_m-u\|_{SBV(\Omega)}=0.
\end{equation*}
\end{lemma}
\begin{proof}
The absolutely continuous part of the relevant gradients can be estimated as follows,
\begin{eqnarray*}
&&\int_\Omega|\nabla^au'_m(x)-\nabla^au(x)dx|\\
&=&\int_{\Omega\setminus\Gamma_m}|\nabla^au'_m(x)-\nabla^au(x)|dx\\
&\leq&\int_{\Omega\setminus\Gamma_m}\left|\int_{\mathbb R^2}\nabla u(x-r_m(x)y)\eta(y)dy-\nabla^au(x)\right|dx   \ \ \ \ \ (=:I_m^1)\\
&&+\int_{\Omega\setminus\Gamma_m}\left|\nabla r_m(x)\int_{\mathbb R^2}\left<\nabla u(x-r_m(x)y),y\right>\eta(y)dx\right|dx\ \ \ \ \ (=:I_m^2)
\end{eqnarray*}

We first estimate $I_m^1$ by consider the regular part and singular part separately, that is,
\begin{eqnarray*}
    I_m^1 &\le& \int_{\Omega\setminus\Gamma_m}\left|\int_{\mathbb R^2}\nabla^au(x-r_m(x)y)\eta(y)dy-\nabla^au(x)\right|dx
+\int_{\Omega\setminus\Gamma_m}\int_{\mathbb R^2}|\nabla^su(x-r_m(x)y)|\cdot\eta(y)dydx \\
& \le & \int_{\Omega\setminus\Gamma_m}\left|\int_{\mathbb R^2}\nabla^au(x-r_m(x)y)\eta(y)dy-\nabla^au(x)\right|dx+C\cdot\int_{\Omega\setminus\Gamma_m}|\nabla^su|
\end{eqnarray*}
The first term in the right hand side vanishes as a consequence of Lemma \ref{lemma_conv_type}, and the second term vanishes by the construction of $\Gamma_m$.

For $I_m^2$, it holds that 
\begin{eqnarray*}
    I_m^2 &:=& 2^{-m}\int_{\Omega\setminus\Gamma_m}\int_{B(0,1)}|\nabla u(x-r_m(x)y)|\cdot\eta(y)dydx \\
    &\le& C\cdot2^{-m}\int_{\Omega\setminus\Gamma_m}|\nabla u|.
\end{eqnarray*}
Since
\begin{equation*}
2^{-m}\int_{\Omega\setminus\Gamma_m}|\nabla u|\leq2^{-m}\|u\|_{SBV(\Omega)}\to0
\end{equation*}
for $u\in SBV(\Omega)$. Combining this estimate with Lemma \ref{lemma_jump} and again Lemma \ref{lemma_conv_type}, we have
\begin{eqnarray*}
&&\|u'_m-u\|_{SBV(\Omega)}\\
&\leq&\|u'_m-u\|_{L^1(\Omega)}+\int_{\Gamma_m}|\nabla^su'_m-\nabla^su|+\int_{\Omega\setminus\Gamma_m}|\nabla^su'_m-\nabla^su|+\int_\Omega|\nabla^au'_m-\nabla^au|\\
&\leq&\|u'_m-u\|_{L^1(\Omega)}+0+2^{-m}+(I_m^1 + I_m^2)\ \ \to0,
\end{eqnarray*}
as $m\to\infty$.
\end{proof}

\bibliographystyle{plain}
\bibliography{ref}
\end{document}